\documentclass[12pt]{amsart}
\usepackage[utf8]{inputenc}
\usepackage{graphicx,amssymb,amsmath,amsthm}
\usepackage{caption}
\usepackage{subcaption}
\usepackage{indentfirst}
\usepackage{cancel}
\usepackage{lipsum}
\usepackage{epstopdf}
\usepackage{epsfig}
\usepackage{hyperref}
\usepackage{todonotes}
\usepackage{enumerate,color}   
\usepackage{tcolorbox}
\usepackage[export]{adjustbox}

\usepackage[top=3cm, bottom=1in, left=1.2in, right=1.2in]{geometry}

\usepackage[overload]{empheq}

\newcommand{\veps}{\varepsilon}

\numberwithin{equation}{section}
\usepackage{soul}
\everymath{\displaystyle}

\theoremstyle{plain}
\newtheorem{theorem}{Theorem}[section]
\newtheorem{lemma}[theorem]{Lemma}
\newtheorem{proposition}[theorem]{Proposition}

\theoremstyle{definition}

\newtheorem*{defi*}{Definition} 
\theoremstyle{remark}
\newtheorem{remark}{Remark}

\theoremstyle{remark}

\title[Peaked singularity in EP system]{Emergence of peaked singularities in the Euler-Poisson system}

\author[J. Bae]{Junsik Bae}
\address[JB]{Department of Mathematical Sciences, Korea Advanced Institute of Science and Technology, Daejeon, 34141, Republic of Korea}
\email{junsikbae@kaist.ac.kr}
  
\author[S.H Moon]{Sang-Hyuck Moon}
\address[SM]{Department of Mathematical Sciences, Ulsan National Institute of Science and Technology, Ulsan, 44919, Republic of Korea}
\email{shmoon0208@unist.ac.kr}

\author[K. Woo]{Kwan Woo}
\address[KW]{The Research Institute of Basic Sciences, Seoul National University, Seoul, 08826, Republic of Korea}
\email{kwanwoo@snu.ac.kr}

\date{\today}



\begin{document}


\begin{abstract} 
We consider the one-dimensional Euler-Poisson system equipped with the Boltzmann relation and  provide the exact asymptotic behavior of the peaked solitary wave solutions near the peak. This enables us to study the cold ion limit of the peaked solitary waves with the sharp range of H\"older exponents.
Furthermore, we provide numerical evidence for $C^1$  blow-up solutions to the pressureless Euler-Poisson system, whose blow-up profiles are asymptotically similar to its peaked solitary waves and exhibit a different form of blow-up compared to the Burgers-type (shock-like) blow-up.
\end{abstract}
 
\maketitle

\section{Introduction}
The dynamics of an electrostatic plasma consisting of ions and electrons are described by the two-species Euler-Poisson system. Based on the physical fact that the mass of electrons is very small compared to that of ions, it is often assumed that the density of electrons follows the so-called Boltzmann relation, $\rho_e=e^\phi$, where $\phi$ is the electric potential. The resulting one-species Euler-Poisson system has not only been proven useful in describing the dynamics of ion waves \cite{Ch,Dav} but also contains rich and interesting mathematical phenomena.

We consider the one-dimensional Euler-Poisson system equipped with the Boltzmann relation
\begin{equation}\label{EP2}
\left\{
\begin{array}{l l}
\partial_{t} \rho + \partial_{x}(\rho v) = 0, \\ 
\rho(\partial_{t} v  + v\partial_{x} v) + \kappa \partial_{x} \rho = -\rho\partial_{x}\phi, \qquad    (t \geq 0, \; x\in\mathbb{R}), \\
-\partial_{x}^2\phi = \rho - e^\phi,
\end{array} 
\right.
\end{equation}
Here, $\rho(t,x)>0$ and $v(t,x)$ are unknown scalar functions for the density and  velocity of ions. The electric potential function $\phi(t,x)$ is determined as the solution to the Poisson equation for given $\rho$. Together with the other two equations, the electric potential is self-consistent. The constant $\kappa \geq 0$ in \eqref{EP2} represents the ratio of the ion temperature to the electron temperature. The system  \eqref{EP2} is referred to as the \textit{pressureless} model when $\kappa=0$ (\textit{cold ion}), and the \textit{isothermal} model when $\kappa >0$, respectively. We refer to \cite{GGPS} for the rigorous derivation of the model \eqref{EP2} from the two-species Euler-Poisson system.   

Among the various types of Euler-Poisson systems, the model \eqref{EP2} is singular in that it describes the emergence of traveling solitary waves \cite{ikezi,Sag,Wa}. It is well-known  \cite{Cor,Satt,Sag} that smooth traveling solitary waves with the speed $c$ exist in the supersonic regime, $\sqrt{1+\kappa} < c < c_\kappa$. At the upper critical speed $c_\kappa$, it admits a \textit{peaked} solitary wave, that is, traveling solitary wave which is not differentiable at the peak. Such peaked solitary waves also arise in the context of water wave models such as the Whitham equation \cite{EMV} and the Camassa-Holm equation \cite{CH1}.

Heuristically, \eqref{EP2} admits traveling solitary waves due to a balance between the nonlinear transport and (weakly) dispersive effect. As a matter of fact, the dispersion relation $\omega$ of \eqref{EP2} around $(\rho,v,\phi)\equiv(1,0,0)$ is given by
\[
\omega(k) = \pm k\sqrt{\kappa + \frac{1}{1+k^2} } = \pm \left( \sqrt{1+\kappa} k - \frac{k^3}{2\sqrt{1+\kappa}} \right) + O(k^5), \quad (k \in \mathbb{R}),
\]
as the frequency $k \to 0$ (i.e, in the long wave regime). Notice that in the moving frame with the speed $\sqrt{1+\kappa}$, the dispersion relation is that of the KdV equation for small $k$.

There are several works concerning the behavior of solutions to \eqref{EP2} near the \textit{ion-sonic} speed $\sqrt{1+\kappa}$.  As $c\searrow \sqrt{1+\kappa}$, the amplitude of solitary wave solutions to \eqref{EP2} converges to zero. In particular, in the long wave scaling,  the Euler-Poisson solitary waves are approximated by the associated KdV solitary waves \cite{BK}:
\begin{equation}\label{KdVlimit}
\veps\cdot \frac{3}{\sqrt{1+\kappa}} \text{sech}^2\left( \sqrt{\sqrt{1+\kappa}/2} \left( x- c t \right)\right), \quad \kappa \geq 0,
\end{equation} 
for sufficiently small $c-\sqrt{1+\kappa}=:\veps>0$. The asymptotic linear stability of the small amplitude solitary waves has been studied in  \cite{BK2} and \cite{Sche}. In terms of the initial value problem, the work of \cite{Guo} rigorously justify the KdV limit of \eqref{EP2}.
To the best of the authors' knowledge, the global existence of smooth solutions to \eqref{EP2} has not yet been studied, with the difficulty arising from the weak dispersive effect in 1D.
We refer to \cite{GP} for the global existence of smooth solutions to the 3D Euler-Poisson system for ion dynamics.

 In this paper, we are interested in the behaviors of the peaked solitary waves, i.e, the solitary wave solutions to \eqref{EP2} with the speed $c=c_\kappa$ for $\kappa \geq 0$. The primary aim of this study is twofold. First, we investigate the exact behavior of the peaked solitary wave solution to \eqref{EP2} near the singularity for each $\kappa \geq 0$. Second, we study the \textit{cold ion limit} ($\kappa \to 0$) of the peaked solitary waves.  

Numerical demonstrations (see Figure 6, \cite{BK2}) indicate that as $c \nearrow c_\kappa$, the smooth solitary waves of \eqref{EP2} converge to some profiles with singularities. The shape of the limiting profile for the pressureless model is completely different from that for the isothermal model. For instance, $\|\rho\|_{L^\infty} \nearrow \infty$ as $c \nearrow c_0$ for  $\kappa=0$, whereas $\|\rho\|_{L^\infty} <\infty$ as $c \nearrow c_\kappa$ for $\kappa>0$. In particular, for $\kappa =0$, $v$ converges to a profile with a cusp-type singularity as $c \nearrow c_0$.

On the other hand, \eqref{EP2} admits \textit{stable} blow-up solutions that are in $C^{1/3}$ at the moment of blow-up  ($C^{\alpha}$-norm blows up for $\alpha>1/3$), see \cite{BKK,BKK2}.
These blow-up solutions are constructed by considering \eqref{EP2} as a perturbation of the Burgers equation,  whose stable self-similar blow-up profile has $C^{1/3}$ regularity \cite{EF}.
This also motivates us to take a closer look at the peaked solitary waves. In fact, our finding in the present paper (Theorem \ref{Thm2}) shows that the peaked solitary waves are not in $C^{1/(2i+1)}$, $i \in \mathbb{N}_+$, which are associated with the Burgers equation \cite{EF}.
A natural question is whether there exists an \textit{open set} of initial data whose blow-up profiles are asymptotically the same as the peaked solitary waves.
In Section \ref{S4}, we discuss on the issue in detail studying numerical solutions of the Euler-Poisson system. We remark that a similar question is also raised in the context of dispersive perturbations of the Burgers equation, see \cite{EMV,KS,OP} and the references therein.

In plasma physics, the pressureless model \eqref{EP2} with $\kappa=0$ is often employed to simplify analysis. Indeed, in terms of the KdV limit of solitary waves, the cold ion limit is regular in the sense of  \eqref{KdVlimit}.
However, the behavior of solutions to the pressureless model generally differs from that of the isothermal model, both qualitatively and quantitatively-for instance, in the case of blow-up solutions \cite{BCK}.
Notably, the singularities of the peaked solitary waves in the isothermal model are milder than those in the pressureless model, making the cold ion limit of the peaked solitary waves a singular limit problem. In fact, we observe a \textit{transition layer} (i.e., the region where the shape of the solution to the isothermal model changes drastically from that to the pressureless model) near the peak, with a thickness of $\kappa^{3/4}$ (see Figure \ref{FigNume00}).
Using precise information about the singularities of the peaked solitary waves, we study the convergence of the peaked solitary waves in the H\"older norm as $\kappa \to 0$ (Theorem \ref{thm3}), providing the sharp range of  H\"older exponents.

\subsection{Main results}
 Plugging the Ansatz $(\rho,v,\phi)(\xi)$ into \eqref{EP2}, where $\xi:=x-ct$,  one obtains  the traveling wave equations
\begin{subequations}\label{TravelEq}
\begin{align}[left = \empheqlbrace\,]
& -c \rho' + (\rho v)'= 0, \label{TravelEq1} \\
& \rho(-cv' + vv') + \kappa \rho' = -\rho \phi',\label{TravelEq2}\\
& - \phi'' = \rho - e^\phi,\label{TravelEq3}
\end{align}
\end{subequations}
where $'$ denotes the derivative in $\xi$. For the far-field condition, we impose
\begin{equation}\label{bdCon+-inf}
\rho \to 1, \;\; v \to 0, \;\; \phi \to 0 \quad \text{as} \quad  \xi  \to  -\infty.
\end{equation}

Let us clearly define the notion of peaked solitary waves. Owing to translation invariance, we set the location of the peak  $\xi =0$  without loss of generality.

\begin{defi*}\label{solitary-def}
$(\rho,v,\phi)(\xi)$ is called a \emph{peaked solitary wave} solution to  \eqref{TravelEq}--\eqref{bdCon+-inf}
if it satisfies  \eqref{TravelEq}--\eqref{bdCon+-inf} for $\xi \in \mathbb{R} \setminus\{0\}$ in the classical sense, and the following hold:
 \begin{enumerate}[(i)]
\item (symmetry)
\begin{equation}  
 \rho(\xi)=\rho(-\xi),\;\; v(\xi)=v(-\xi), \;\; \phi(\xi)=\phi(-\xi) \quad  \text{for} \quad \xi \neq 0, \label{Symmetric} 
\end{equation}
\item (monotonicity)
\begin{equation}\label{mono}
 \rho'(\xi), \;\; v'(\xi), \;\; \phi'(\xi) > 0  \quad  \text{for} \quad \xi \in (-\infty,0),
\end{equation}
\item (singularity at the peak) 
\begin{equation}\label{peak}
\text{at } \xi=0, \text{the right derivative of } (\rho,v)(\xi)  \text{ is not equal to the left derivative}. 
\end{equation}
\end{enumerate}
\end{defi*}  

Due to \eqref{bdCon+-inf}, the peaked solitary waves $(\rho,v,\phi)$ satisfy
\begin{equation*}\label{SignOfSols}
\rho(\xi)>1, \;\; v(\xi)>0, \;\; \phi(\xi)>0 \quad \text{for} \quad \xi \neq 0. 
\end{equation*}
If \eqref{peak} is not satisfied, the solution is called a smooth solitary wave. 

Before discussing the existence of solitary waves, we first introduce some preliminaries. For  $\kappa>0$, let $c_\kappa$ be a unique root of   the equation
\begin{equation}\label{Aux3 lem-1}
z^{\kappa}\left[(z - \sqrt{\kappa})^2 + 1 \right]  =
\exp \left( \left(z^2 - \kappa \right)/2\right)  \kappa^{\kappa/2}
\end{equation}
on the interval  $(\sqrt{\kappa},\infty)$. It holds that
$c_{\kappa}>\sqrt{1+ \kappa}$ (see Appendix). 
Let $c_0>1$ be a unique positive root of
\begin{equation}\label{Eq z0 Cold-1}
z^2+1 = \exp(z^2/2).
\end{equation}  

For $c>0$ satisfying $\sqrt{1+\kappa} < c < c_\kappa$ ($1 < c <  c_0$ when $\kappa=0$), the problem \eqref{TravelEq}--\eqref{bdCon+-inf} admits a unique  smooth  solution $(\rho,v,\phi)(\xi) \in C^\infty(\mathbb{R})$ satisfying  \eqref{Symmetric}--\eqref{mono}, i.e., a smooth solitary wave solution, \cite{Cor,Satt,Sag}.

For $c=c_\kappa$, \eqref{TravelEq}--\eqref{bdCon+-inf} admit a peaked solitary wave $(\rho,v,\phi) \in C^{\infty}(\mathbb{R}\setminus \{0\})$.  The existence of the peaked solitary waves itself is addressed in \cite{Cor} and \cite{Satt}.
Nevertheless, we restate it below to suit our purpose and present the detailed proof in Section 2 since it contains important information which will be used in the proof of our main results.\footnote{Theorem 5.1 in \cite{Cor}, where the isothermal Euler-Poisson system is considered, states that the solitary wave solution for $c=c_\kappa$ is smooth. To be more precise,  it is in fact Lipschitz continuous at the peak, as stated in \eqref{Thm2_2}.  For the pressureless case, see Theorems 4.1--4.2 of \cite{Satt}.}
Let 
\[
\rho^\ast:= \sup_{\xi\in \mathbb{R}\setminus\{0\}} \rho, \quad  v^\ast:= \sup_{\xi\in \mathbb{R}\setminus\{0\}} v, \quad \phi^\ast:= \sup_{\xi\in \mathbb{R}\setminus\{0\}} \phi.
\]
\begin{proposition}
	\label{MainThm} 
For $\kappa>0$, let $c_\kappa$ be a unique root of \eqref{Aux3 lem-1} satisfying $c_\kappa>\sqrt{1+\kappa}$. For  $\kappa=0$, let $c_0$ be a unique root of \eqref{Eq z0 Cold-1} satisfying $c_0>1$. 
Consider the system \eqref{TravelEq}--\eqref{bdCon+-inf} with  $c=c_\kappa$.  Then, it admits a unique peaked solitary wave solution $(\rho,u,\phi)$  satisfying \eqref{Symmetric}--\eqref{peak}. The solution $(\rho,u,\phi)$ exponentially decays to $(1,0,0)$ as $|\xi|\to \infty$. Furthermore,   when $\kappa=0$, we have
\begin{equation}\label{Maxk0}
\rho^\ast=+\infty, \quad v^\ast=c_0, \quad  \phi^\ast=c_0^2/2;
\end{equation}
when $\kappa>0$, we have
\begin{equation}\label{Maxk}
\rho^\ast=c_\kappa/\sqrt{\kappa}, \quad v^\ast=c_\kappa-\sqrt{\kappa}, \quad \phi^\ast=H(\rho^\ast),
\end{equation}
where $H$ is defined in \eqref{Def_H}.
\end{proposition}

Our first main theorem provides the exact behaviors of singularities of the peaked solitary waves to the Euler-Poisson system. 
\begin{theorem}\label{Thm2}
For $c=c_\kappa$, let $(\rho,v,\phi)$ be the solution to \eqref{TravelEq}--\eqref{bdCon+-inf}  satisfying \eqref{Symmetric}--\eqref{peak}.
Then, the following hold: 
\begin{enumerate}
\item When $\kappa=0$, 
\begin{equation}\label{Thm2_1} 
\left\{\begin{array}{l l}
& \lim_{\xi \to 0}\frac{\phi^\ast-\phi}{|\xi|^{4/3}}=A_0:=\frac{1}{2}\left(3 \sqrt{\frac{c_0}{2}} \right)^{4/3}, \quad  \lim_{\xi \to 0}\frac{|\phi'|}{|\xi|^{1/3}}=\frac{4}{3}A_0,  \\
& \lim_{\xi \to 0}|\phi''||\xi|^{2/3}=\frac{4}{9}A_0 = \lim_{\xi \to 0} \rho|\xi|^{2/3}, \quad \lim_{\xi \to 0}\frac{v^\ast-v}{|\xi|^{2/3}}=\sqrt{2A_0}.
\end{array} \right.
\end{equation}  
\item  When $\kappa>0$, 
\begin{subequations}\label{Thm2_2}
\begin{align}[left = \empheqlbrace\qquad]
& \lim_{\xi \to  0}\frac{\phi^\ast - \phi}{|\xi|^2} = \lim_{\xi \to 0} \frac{-\phi'}{2|\xi|} = \lim_{\xi \to 0}\frac{-\phi''}{2} = \frac{1}{2} (\rho^\ast-e^{\phi^\ast})>0, \label{Thm2_3} \\
&  \lim_{\xi \to 0} \frac{\rho^\ast -\rho}{|\xi|} = \frac{(\rho^\ast)^2}{c_\kappa} \left( \lim_{\xi \to 0} \frac{v^\ast-v}{|\xi|} \right) = \sqrt{\frac{\rho^\ast - e^{\phi^\ast}}{-\frac{dh}{d\rho}(\rho^\ast)}}, \label{Thm2_4}
\end{align}
\end{subequations}
where $h(\rho)$ is given by \eqref{Def_h}.
\end{enumerate}
\end{theorem}

When $\kappa=0$, we see that $\rho-1 \in L^p(\mathbb{R})$ for $1\leq p<3/2$ and $v$ behaves like $|\xi|^{2/3}$ near $\xi=0$.
Hence, the peaked solitary wave solution of \eqref{TravelEq}--\eqref{bdCon+-inf} (and \eqref{EP2}) can be defined by  the distributional sense near $\xi = 0$.
When $\kappa>0$, the peaked solitary wave solution is Lipschitz continuous at $\xi=0$. The relatively relaxed singularities at $\xi=0$ is due to the regularizing effect of the pressure term.  
 
In Figures \ref{FigNume1} and \ref{FigNume2}, we present numerical demonstrations of the asymptotic shapes of the singularities. MATLAB’s ode89 is used to solve the initial value problem for the system of ODEs \eqref{ODE_n_E} with $c=c_\kappa$. The initial conditions are suitably chosen using the first integral \eqref{1st Int}. \\  

From Theorem \ref{Thm2}, we see that the behaviors of the peaked solitary waves near $\xi=0$ are not regular in the cold ion limit (as $\kappa \to 0$).
Indeed, there is a \textit{transition} layer near $\xi=0$ whose length tends to zero as $\kappa \to0$ (see Remark \ref{Rem1} and Figure \ref{FigNume00}).
Our second main result provides the cold ion limit of the peaked solitary waves in a strong sense with an optimal H\"older regularity. 
 
Throughout the paper, we will use the subscript $\kappa$ to indicate the dependence on $\kappa \geq 0$ when necessary.
\begin{theorem}	\label{thm3}
For any $\alpha\in(0,1/3)$, $\beta\in(0,2/3)$ and $p \in [1, 3/2)$, it holds that  
\begin{equation}\label{Eq_Thm3}
\phi_\kappa \to \phi_0 \quad \text{ in } C^{1,\alpha}(\mathbb{R}),
\end{equation}
\begin{equation}\label{Eq_Thm3_2}
\rho_\kappa - 1 \to \rho_0 -1 \quad \text{ in } L^p(\mathbb{R}),
\end{equation}
and
\begin{equation}\label{Eq_Thm3_3}
v_\kappa \to v_0 \quad \text{ in } C^{\beta}(\mathbb{R})
\end{equation}
as $\kappa \to 0$.
\end{theorem}

We remark that the limit \eqref{Eq_Thm3} does not hold for $\alpha=1/3$.
Indeed, since $\phi_\kappa'(0)=0$ for $\kappa\geq 0$ by Theorem \ref{Thm2}, we have for each $\kappa>0$,
\begin{equation}\label{Eq2_Thm3}
\frac{|(\phi_\kappa'-\phi_0')(\xi) - (\phi_\kappa'-\phi_0')(0)|}{|\xi|^{1/3}} \geq \frac{|\phi_0'(\xi)|}{|\xi|^{1/3}} - \frac{|\phi_\kappa'(\xi)|}{|\xi|^{1/3}}.
\end{equation}
By taking the limsup on both sides of \eqref{Eq2_Thm3}, and then using Theorem \ref{Thm2}, we see that
\[
\begin{split}
\sup_{\xi \neq 0 }\frac{|(\phi_\kappa'-\phi_0')(\xi) - (\phi_\kappa'-\phi_0')(0)|}{|\xi|^{1/3}}   \geq  \limsup_{\xi \to 0}\frac{|\phi_0'(\xi)|}{|\xi|^{1/3}},
\end{split}
\]
where the right-hand side is independent of $\kappa>0$ and strictly positive.

\begin{remark}[Transition layer]
	\label{Rem1}
Roughly speaking, from the perspective of boundary layer problems \cite{GHJT}, $(\rho_0,v_0,\phi_0)$ can be considered as the outer solution in that it describes the behavior of $(\rho_\kappa,v_\kappa,\phi_\kappa)$ in the far-field (outside of the layer near $\xi=0$) for small $\kappa \geq 0$.
The thickness of the boundary layer can be obtained as follows.
Using $\rho^\ast_\kappa = c_\kappa/\sqrt{\kappa}$, \eqref{Def_H}, and Lemma \ref{z_K est}-(4), it is straightforward to check that $\rho^\ast - e^{\phi^\ast} =  c_0/\sqrt{\kappa}  + O(1)$ as $\kappa \to 0.$ Hence, by Theorem \ref{Thm2}, for sufficiently small $\kappa>0$, it holds that
\[
\phi_0^\ast - \phi_0 \simeq |\xi|^{4/3}, \quad 
\phi_\kappa^\ast - \phi_\kappa \simeq \kappa^{-1/2}\xi^2  \quad \text{ as } \quad \xi \to 0. 
\]
Balancing the orders of $|\xi|^{4/3}$ and $\kappa^{-1/2}\xi^2$, we see that $\kappa^{3/4}$ is the thickness of the boundary layer at $\xi = 0$ (see also Figure \ref{FigNume00}). Since our problem is not a boundary layer problem, however,  we will refer to it as the \textit{transition} layer. In Proposition \ref{prop3.4}, we carefully estimate the behavior of the peaked solitary waves inside the transition regime.
For the boundary layer (\textit{plasma sheath}) problems arising in the \textit{quasi-neutral} limit (the Debye length tends to zero) of the Euler-Poisson system for ion dynamics, we refer to \cite{GHR,JKS1,JKS2,JKS3}.
\end{remark}

To the best of our knowledge, there are no results concerning the stability of the peaked solitary waves in the Euler-Poisson system.
Since these waves only exist at the critical traveling speed,  they might appear insignificant at first glance.
Surprisingly, our numerical observations reveal a possible connection between the peaked solitary waves and the $C^1$ blow-up phenomena in the Euler-Poisson system, particularly within the pressureless model.
We will discuss this issue in more detail in Section \ref{S4}.

 \begin{figure}[t]
\begin{center}
\includegraphics[width=0.55\linewidth]{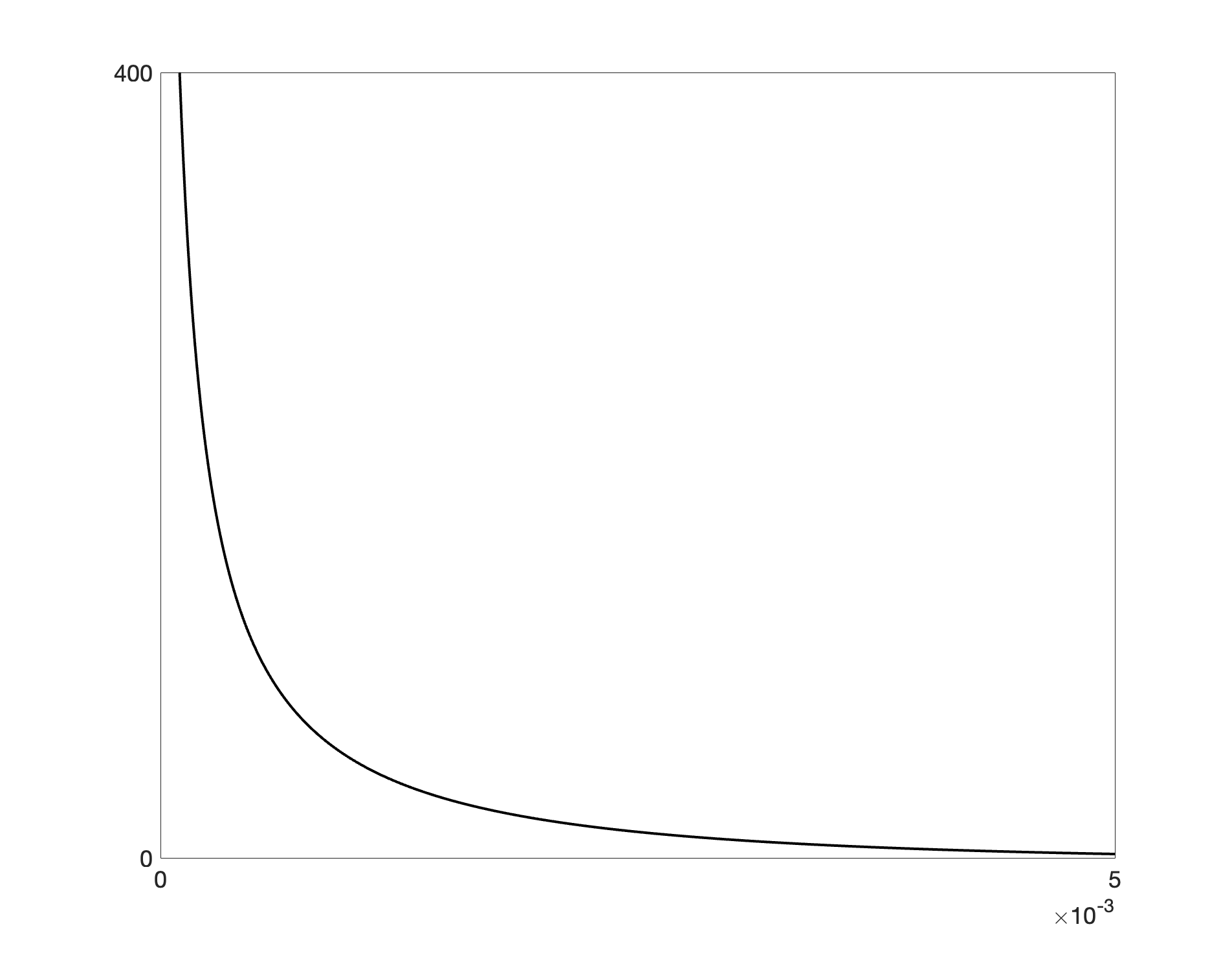}
\end{center}
\caption{A numerical plot of the inner solution $\rho_0(\xi)-\rho_\kappa(\xi)$ on $\xi\in[0,0.005]$ for $\kappa=0.001$. The thinkness of the transition layer is $(0.001)^{3/4} = 0.0056\cdots$, as expected in Remark \ref{Rem1}. }
\label{FigNume00}
\end{figure}

\begin{figure}[h]
\begin{tabular}[]{@{}c@{}@{}c@{}} 
\resizebox{63mm}{!}{\includegraphics{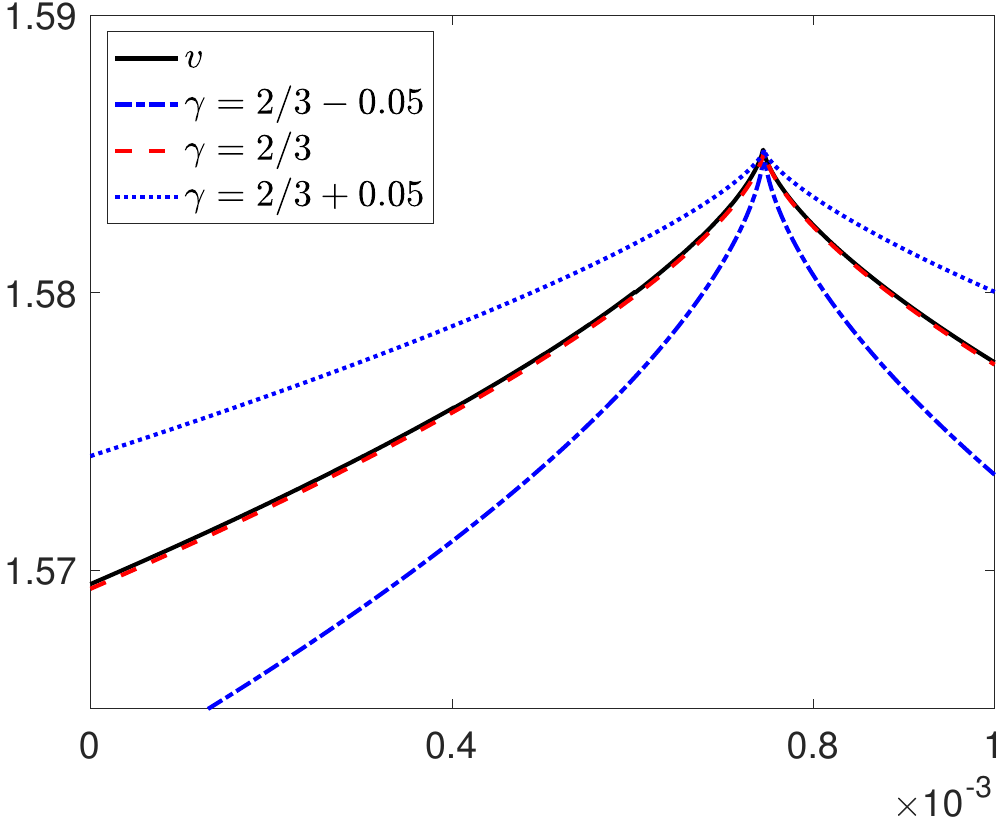}}  & \resizebox{63mm}{!}{\includegraphics{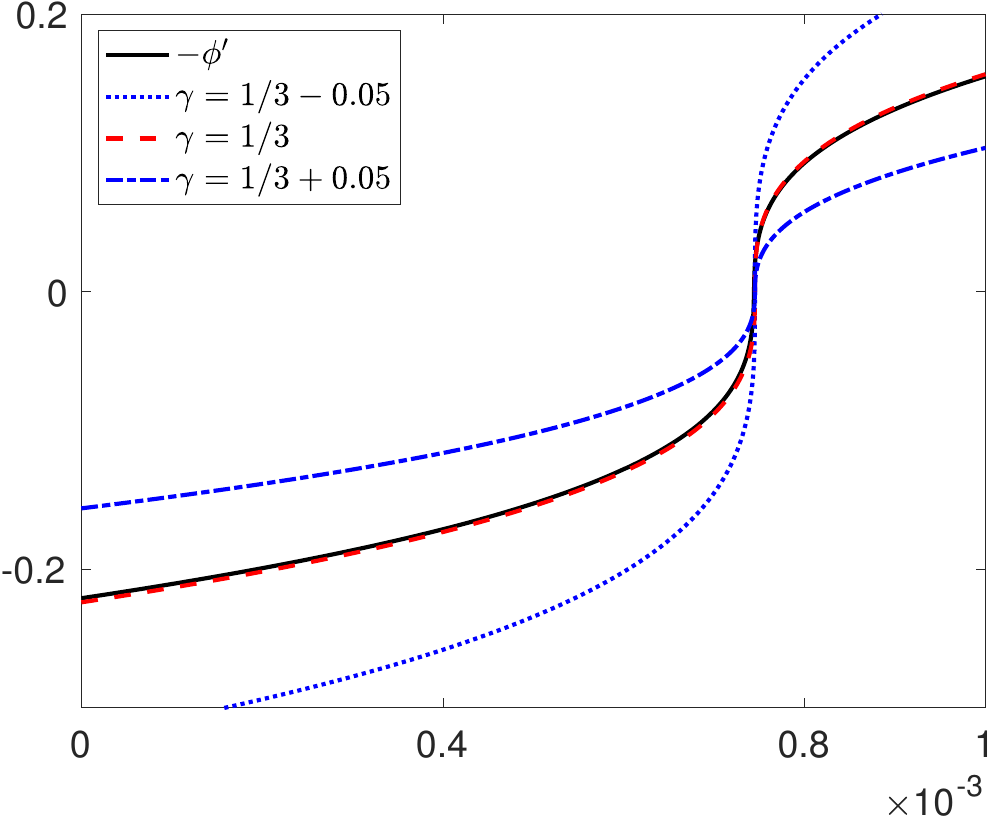}}  
\end{tabular}
\caption{Asymptotic behaviors of solitary waves near the peaks for the pressureless case ($\kappa=0$) at $c=1.58519 < c_0 = 1.585201\cdots$. $\Delta \xi = 10^{-6}.$  Left: $v$ is compared with $v^\ast - \sqrt{2A_0}|\xi|^{\gamma}$, where $A_0$ is given in Theorem \ref{Thm2}.
Right: $-\phi'$ is compared with $\textstyle{\frac{4A_0}{3}\xi^\gamma}$. } 
\label{FigNume1}
\end{figure}

\begin{figure}[h]
\begin{tabular}[]{@{}c@{}@{}c@{}} 
\resizebox{63mm}{!}{\includegraphics{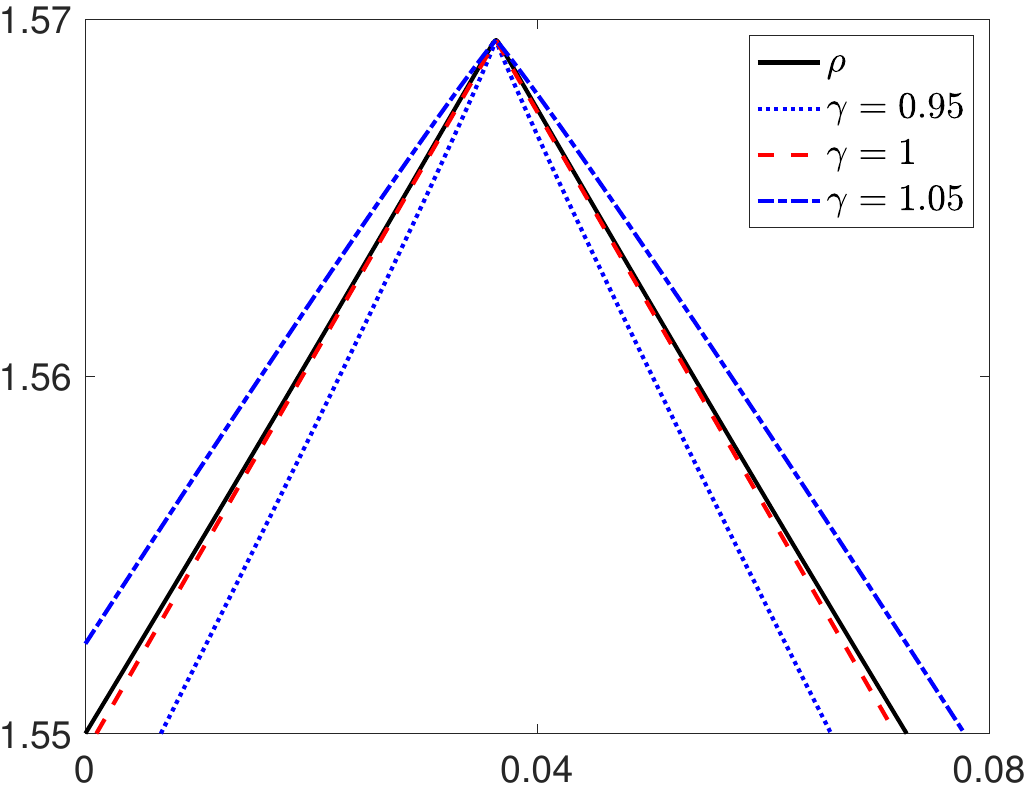}}  & \resizebox{65mm}{!}{\includegraphics{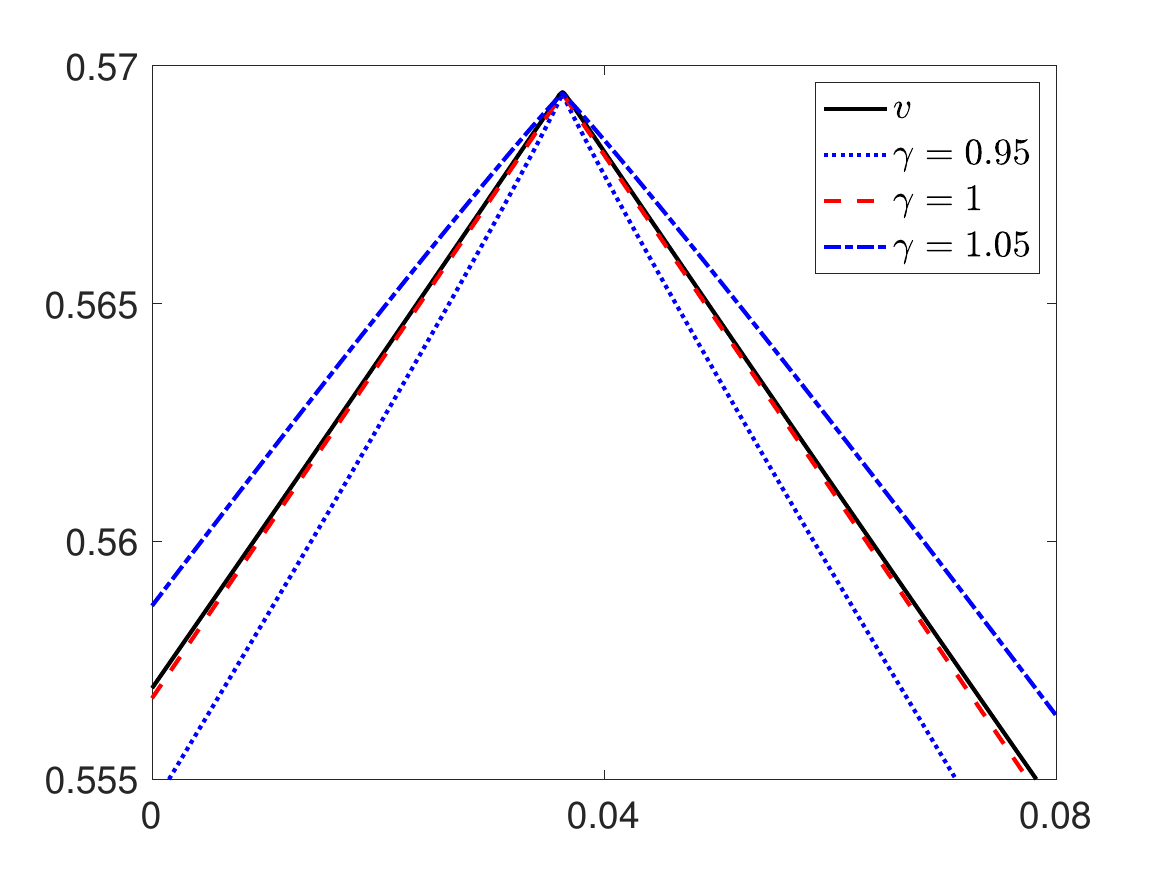}}  
\end{tabular}
\caption{Asymptotic behaviors of solitary waves near the peaks for the isothermal case ($\kappa=1$) at $c=\sqrt{2}+0.15527028$. $\Delta \xi = 10^{-6}.$ For $\kappa=1$, $c_\kappa = \sqrt{2} + 0.1552702843\cdots$.  Left: $\rho$ is compared with $\textstyle{\rho^\ast - \sqrt{-\left(\rho^\ast - e^{\phi^\ast}\right)/\frac{dh}{d\rho}(\rho^\ast)}|\xi|^\gamma}$.   Right: $v$ is compared with $\textstyle{v^\ast - \frac{c_\kappa}{(\rho^\ast)^2}\sqrt{-\left(\rho^\ast - e^{\phi^\ast}\right)/\frac{dh}{d\rho}(\rho^\ast)}|\xi|^\gamma}$. }
\label{FigNume2}
\end{figure}

\subsection{Strategy of the proof}
The proof of Proposition \ref{MainThm} is based on a  phase plane analysis of the first-order ODE system \eqref{ODE_n_E}, which is reduced from \eqref{TravelEq}--\eqref{bdCon+-inf}, as in \cite{BK, Cor}.
It is worth mentioning that the solutions to \eqref{TravelEq}--\eqref{bdCon+-inf} have the following relation:
\[
\phi = H(\rho) := \frac{c^2}{2}\left(1-\frac{1}{\rho^2} \right) - \kappa \log\rho.
\]
When $\kappa = 0$, $H^{-1}(\phi)$ can be explicitly expressed, and the problem is readily reduced to a second-order ODE for $\phi$; See \eqref{11} and \eqref{111_0}. However, for $\kappa>0$, $H(\rho)$ is no longer monotonic on $\rho \in (1, \infty)$, meaning we need to restrict the interval where $H$ is invertible. Furthermore, $H^{-1}(\phi)$ is not explicit when $\kappa > 0$.
Hence, analyzing \eqref{ODE_n_E} is more convenient for the existence proof.

 To establish the cold ion limit of the peaked solitary waves (Theorem \ref{thm3}), we apply Arzela-Ascoli's theorem after deriving a uniform bound of $\|\phi_\kappa\|_{C^{1,\alpha}}$.  For this purpose, we utilize the second-order ODE \eqref{12_0809} for $\phi_\kappa$ rather than the first-order ODE system \eqref{ODE_n_E}.
Since our argument relies on a compactness argument, the uniqueness of the peaked solitary wave $\phi_0$ plays a key role in specifying the limit of $\phi_\kappa$.
Finally, using the relation among $\phi$, $\rho$, and $v$ in \eqref{TravelEqB}, we show the cold ion limit $\rho_\kappa \to \rho_0$ and $ v_\kappa \to v_0$.
 
The proof of the uniform bound of $\|\phi_\kappa\|_{C^{1,\alpha}}$ is divided into a sequence of lemmas and propositions. The main difficulty in obtaining the uniform bound arises from the different behavior of the peaked solitary waves inside and outside the transition regime, whose length tends to zero as $\kappa \to 0$.
The thickness of the transition layer is affected by a nonlinear term $H_\kappa^{-1}$ in the second-order ODE \eqref{12_0809} for $\phi_\kappa$, hence  we need to investigate $H_\kappa^{-1}$ carefully.
In Lemmas \ref{H_K conv} and \ref{lem3.3}, we establish some preliminaries concerning the  behavior of $H_\kappa^{-1}(\phi)$.

We partition the domain $\mathbb{R}$ into three regions:
(i) the interval of transition $I_\kappa \ni 0$, where the solitary waves $\phi_\kappa$ behave as in \eqref{Thm2_2};
(ii) the region $J_\kappa \setminus I_\kappa$, where $\phi_\kappa$ behaves as in \eqref{Thm2_1};
(iii) the far-field (or exponentially decaying) region $\mathbb{R} \setminus J_\kappa$, which contains the unique inflection point $\xi_\kappa$ where the concavity of the solitary waves $\phi_\kappa$ changes.
We analyze the behavior of $\phi_\kappa$ on $I_\kappa$ and $J_\kappa \setminus I_\kappa$ in Propositions \ref{prop3.4} and \ref{prop3.4_1}, respectively.
Lastly, in Proposition \ref{prop3.6}, we obtain the uniform exponential decay of $\phi_\kappa$ on $\mathbb{R} \setminus J_\kappa$ by showing that $\xi_\kappa$ is of order 1. We remark that $\xi_\kappa$ is the point where the center of the first-order ODE system \eqref{ODE_n_E} is attained (compare \eqref{Aux 6} and \eqref{step3eq1}, and see Figure \ref{FigNumeric}).\\

The paper is organized as follows.
In Section 2, we reduce the traveling wave equations \eqref{TravelEq}--\eqref{bdCon+-inf} to a first-order ODE system and prove Proposition \ref{MainThm} by the phase-plane analysis. 
In Section 3, we prove the asymptotic behavior of the peaked solitary waves near $\xi=0$  for fixed $\kappa \ge 0$ (in subsection 3.1) and the cold ion limit ($\kappa \to 0$) of the peaked solitary waves (in subsection 3.2).
Section 4 is devoted to the numerical simulation results, which give evidence for the blow-up phenomena different from the Burgers-type singularity formation.
In Section 5, we prove an auxiliary lemma about the traveling wave speeds $c_\kappa$, $\kappa \ge 0$.

\section{Existence of peaked solitary waves}
In this section, we carry out a phase plane analysis. To this end, we first reduce \eqref{TravelEq} to a first-order ODE system for $(\rho,-\phi')$.
\subsection{Reduction to a first-order ODE system}\label{reduction}
Suppose that $(\rho,v,\phi)(\xi)$ is a smooth solution to \eqref{TravelEq}  satisfying $\rho>0$ for $\xi\in (-\infty,0)$ and $(\rho,v,\phi)\to (1,0,0)$ as $\xi \to -\infty$. We integrate \eqref{TravelEq1}--\eqref{TravelEq2} from $-\infty$ to $\xi$. Then, we have 
\begin{subequations}\label{TravelEqB}
\begin{align}[left = \empheqlbrace\,]
& v=c\left( 1-\frac{1}{\rho} \right), \label{TravelEqB1} \\ 
& \phi = H(\rho):= \frac{c^2}{2}\left(1 - \frac{1}{\rho^2} \right)  - \kappa\ln \rho.  \label{Def_H}
\end{align}
\end{subequations}
By taking the derivative of \eqref{Def_H}, we obtain the reduced ODE system 
\begin{subequations}\label{ODE_n_E}
\begin{align}[left = \empheqlbrace\,]
& -h(\rho) \rho' = E, \label{ODE_n_E1} \\
&   E' = \rho- e^{H(\rho)},\label{ODE_n_E2}
\end{align}
\end{subequations}
where $E:=-\phi'$ and 
\begin{equation}\label{Def_h}
h(\rho):= \frac{dH(\rho)}{d\rho} = \frac{c^2}{\rho^3} - \frac{\kappa}{\rho}.
\end{equation}
The far-field condition is given by
\begin{equation}\label{bdCon n E}
\rho \to 1, \quad  E \to 0 \quad \text{as}\quad  \xi \to -\infty.
\end{equation}  
The solution to \eqref{TravelEq} is obtained from the solution to \eqref{ODE_n_E} via the relations \eqref{TravelEqB}.

Multiplying \eqref{ODE_n_E2} by $E$, and then using \eqref{ODE_n_E1} and \eqref{Def_h}, we get
\begin{equation}\label{Aux 4} 
\frac{1}{2}\left( E^2\right)' = -\rho h(\rho) \rho' + e^{H(\rho)}h(\rho) \rho'  = g(\rho)',
\end{equation}
where  $g(\rho):= c^2/\rho + \kappa \rho + e^{H(\rho)}$. Integrating \eqref{Aux 4} in $\xi$, we obtain a first integral $\Psi(\rho,E)$, a conserved quantity along the solution to \eqref{ODE_n_E}, and from \eqref{bdCon n E}, we have 
\begin{equation}\label{1st Int}
\Psi(\rho,E) = -E^2/2 + g(\rho) = g(1).
\end{equation}  

Using Lemma \ref{z_K est}, a direct calculation yields the following lemma:
\begin{lemma}\label{Lemma1}
\begin{enumerate}
\item When $\kappa=0$, $c=c_0>1$ satisfies \eqref{Eq z0 Cold-1} if and only if $\lim_{\rho \to +\infty} g(\rho) = g(1)$ and $c>0$. 
\item When $\kappa >0$, $c=c_\kappa>\sqrt{1+\kappa}$ satisfies \eqref{Aux3 lem-1} if and only if 
\begin{equation}\label{Cond g}
g(1) = g(c/\sqrt{\kappa}) \quad \textrm{and}  \quad c >  \sqrt{\kappa}.
\end{equation}
\end{enumerate}
\end{lemma}

\subsection{Stationary points}\label{Stationary}

We claim that   the system \eqref{ODE_n_E} has only two stationary points $(\rho,E)=(1,0)$ and $(\rho,E)=(\hat{\rho},0)$, where $\hat{\rho}$ is a solution to 
\begin{equation}\label{Aux 9}
\rho  = e^{H(\rho )}.
\end{equation}
Indeed, for $\rho>0$, \eqref{Aux 9} holds if and only if  $l(\rho):=\ln \rho - H(\rho)=0$. Since $dl/d\rho = (1+\kappa)/\rho - c^2/\rho^3$, we see that  $l(\rho)$ strictly decreases on $(0,\rho_1)$ and strictly increases on $(\rho_1,\infty)$, where
\begin{equation*}\label{Def n_1}
1 < \rho_1: = \frac{c}{\sqrt{1+\kappa}}, \quad (\kappa \geq 0).
\end{equation*}
Here, $1 < \rho_1$ holds true since $c=c_\kappa > \sqrt{1+\kappa}$ for $\kappa\geq 0$. 

On the other hand, from the facts that $l(1) = 0$ and $\lim_{\rho \to \infty}l(\rho) = \infty$, it follows that $l(\rho)$ has only two zeros $\rho=1$ and $\rho=\hat{\rho}$, and it holds that 
\begin{equation}\label{Aux 10}
1 < \frac{c}{\sqrt{1+\kappa}} < \hat{\rho}, \quad   (\kappa \geq 0).
\end{equation}

In what follows, we record a few facts for a later purpose. The above analysis yields that for $\kappa \geq 0$, 
\begin{equation}\label{Aux 6}
\left\{
\begin{array}{l l}
\rho  <  e^{H(\rho)} \quad \text{for} \; \rho \in (1,\hat{\rho}), \\ 
\rho  >  e^{H(\rho)} \quad \text{for} \; \rho \in (\hat{\rho},\infty).
\end{array} 
\right.
\end{equation}
From the definition of $g$ (see \eqref{Aux 4}), we have
\begin{equation}\label{gPrime}
\begin{split}
\frac{dg}{d\rho}   = -h(\rho)\left(\rho-  e^{H(\rho)}\right).
\end{split}
\end{equation}
From \eqref{Def_h}, we see that when $\kappa=0$, 
\begin{equation}\label{sign h cold}
h(\rho)>0 \;\; \text{for} \;\; \rho>0,
\end{equation}
and when $\kappa>0$, 
\begin{equation}\label{sign h}
h(\rho)
\left\{
\begin{array}{l l}
=0 \;\; \text{for} \;\; \rho=c/\sqrt{\kappa}, \\ 
<0 \;\; \text{for} \;\; \rho>c/\sqrt{\kappa}, \\
>0 \;\; \text{for} \;\; 0<\rho<c/\sqrt{\kappa}.
\end{array}
\right.
\end{equation}

\begin{lemma}
	\label{Lemma2}
For $\hat{\rho}$ satisfying \eqref{Aux 9}, the following holds:
\begin{enumerate}
\item When $\kappa=0$,  for $c=c_0$, $g(\rho)$ strictly increases on $(1,\hat{\rho})$ and strictly decreases on $(\hat{\rho},\infty)$.
\item When $\kappa >0$,  for $c=c_\kappa$, it holds that 
\begin{equation}\label{hat k}
1< \hat{\rho} < c_\kappa / \sqrt{\kappa}.
\end{equation}
Moreover, $g(\rho)$ strictly increases on $(1,\hat{\rho})$ and strictly decreases on $(\hat{\rho},c/\sqrt{\kappa})$.
\item For $c=c_\kappa$ with $\kappa \geq 0$,  $(\rho,E)=(1,0)$ is a saddle and $(\rho,E) = (\hat{\rho},0)$ is a center.
\end{enumerate}
\end{lemma}

\begin{proof}
\textit{(1)} It directly follows from  \eqref{Aux 6}, \eqref{gPrime}, and \eqref{sign h cold}.

\textit{(2)} We first show \eqref{hat k}.  $1 < \hat{\rho}$ holds by \eqref{Aux 10}. Suppose that $\hat{\rho} \geq c_\kappa /\sqrt{\kappa}$. Then, together with \eqref{sign h} and \eqref{Aux 6}, it follows from \eqref{gPrime} that $g(\rho)$ strictly increases on $(1,c_\kappa/\sqrt{\kappa})$, which contradicts \eqref{Cond g}. Hence, \eqref{hat k} holds true. Using \eqref{hat k}, together with \eqref{Aux 6}, \eqref{gPrime}, and \eqref{sign h}, we see that $g(\rho)$ strictly increases on $(1,\hat{\rho})$ and strictly decreases on $(\hat{\rho},c/\sqrt{\kappa})$. 

\textit{(3)}  The Jacobian matrix of the system \eqref{ODE_n_E} is  
\begin{equation*}\label{jacobian}
J(\rho,E):=\left( 
\begin{array}{cc}
\frac{E}{[h(\rho)]^2}\frac{dh(\rho)}{d\rho} & \frac{-1}{h(\rho)} \\
   1-h(\rho)e^{H(\rho)}   & 0
\end{array}
\right).
\end{equation*}
Since $c_\kappa>\sqrt{1+\kappa}$ for $\kappa \geq 0$, the eigenvalues $\lambda$ of $J(1,0)$ satisfy
\begin{equation*}\label{0901}
\lambda^2= -\frac{1 - h(1)e^{H(1)}}{  h(1)} =  \frac{ c^2 - (1+\kappa)}{ c^2-\kappa} > 0.
\end{equation*}
Hence, $(\rho,E)=(1,0)$ is a saddle point for $\kappa \geq 0$. 

On the other hand, since $\hat{\rho}$ satisfies \eqref{Aux 9}, we obtain from \eqref{Aux 10} that
\begin{equation}\label{AuxCalcul1}
\begin{split}
1 - h(\hat{\rho})e^{H(\hat{\rho})}
& = 1 + \kappa - \left( \frac{c}{\hat{\rho}} \right)^2 >0.
\end{split}
\end{equation}
From \eqref{sign h cold}, \eqref{sign h}, \eqref{hat k}, and \eqref{AuxCalcul1}, it is straightforward to check that   $J(\hat{\rho},0)$ has two purely imaginary eigenvalues, and thus $(\hat{\rho},0)$ is a center for $\kappa \geq 0$.  We finish the proof.
\end{proof}

\subsection{Proof of Proposition \ref{MainThm}}
From \eqref{ODE_n_E} and the analysis in subsection \ref{Stationary}, one can check that when $\kappa=0$,
\begin{equation}\label{Direction}
\left\{
\begin{array}{l l}
\rho' > 0  \text{ and }  E'<0 \quad \text{for } \rho \in (1,\hat{\rho}) \text{ and } E<0, \\ 
\rho' > 0  \text{ and }  E' > 0  \quad  \text{for } \rho \in (\hat{\rho},\infty) \text{ and } E<0, \\
\rho' < 0  \text{ and }  E' > 0  \quad  \text{for } \rho \in (\hat{\rho},\infty) \text{ and } E > 0, \\
\rho' < 0  \text{ and }  E' < 0  \quad \text{for } \rho \in (1,\hat{\rho}) \text{ and } E > 0.
\end{array} 
\right.
\end{equation}
When $\kappa>0$, \eqref{Direction} with $(\hat{\rho},\infty)$ replaced by $(\hat{\rho},c_\kappa/\sqrt{\kappa})$ hold. See Figure \ref{FigNumeric}.

\begin{figure}[h]
\begin{tabular}{cc}
\resizebox{60mm}{!}{\includegraphics{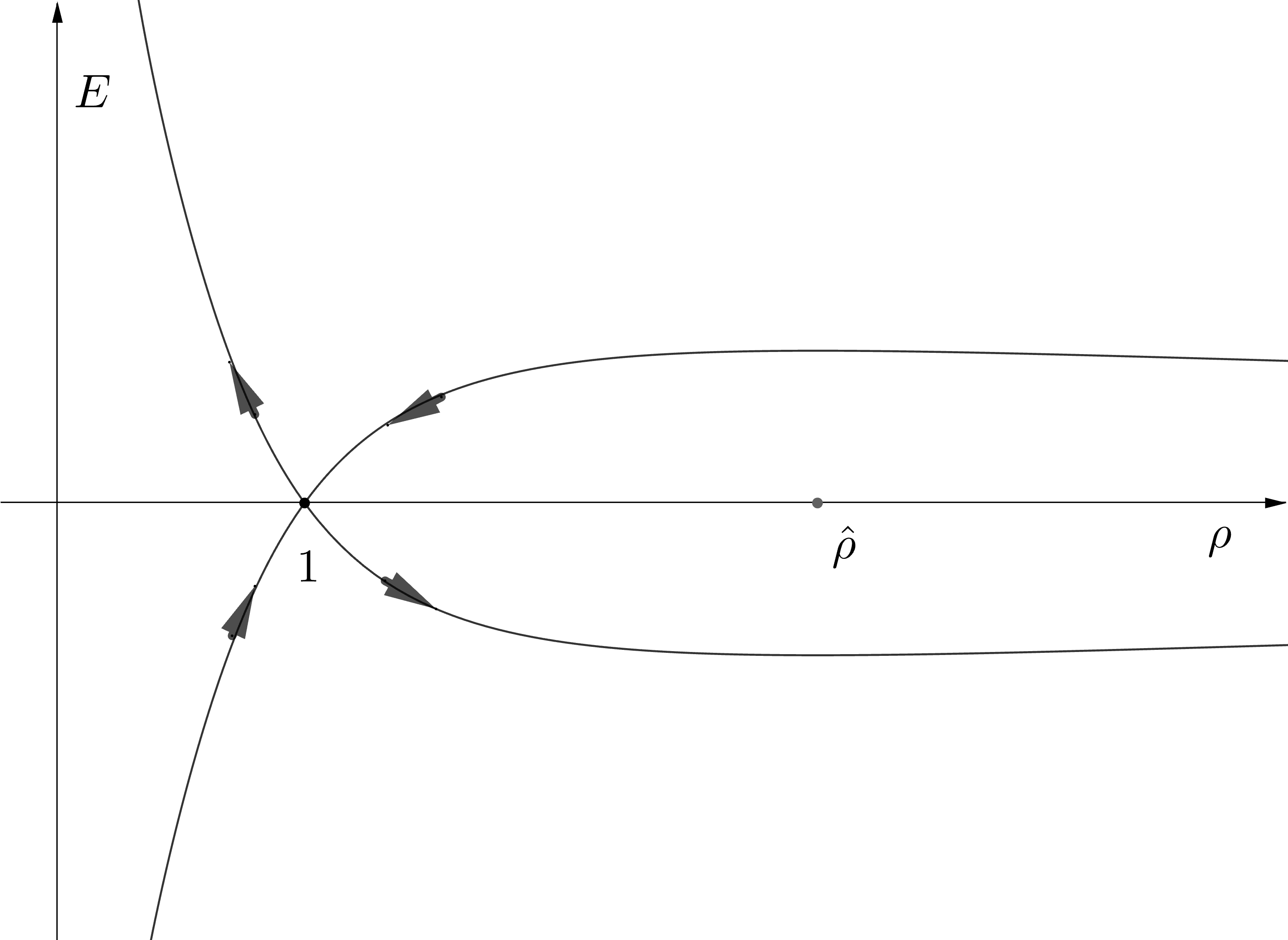}}  & \resizebox{60mm}{!}{\includegraphics{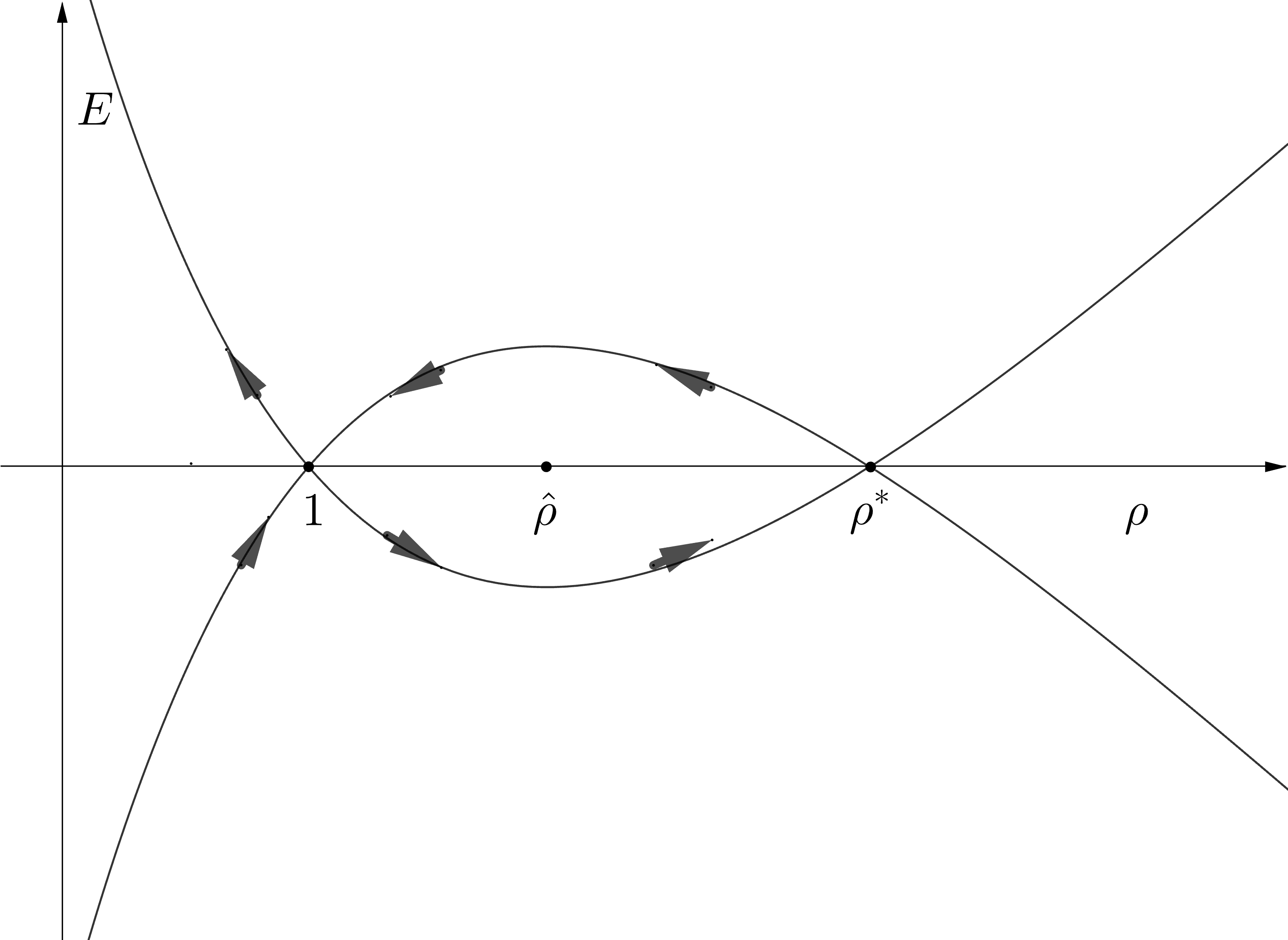}} \\
(a) $\kappa=0$  & (b) $\kappa > 0$
\end{tabular}
\caption{Trajectory curves}
\label{FigNumeric}
\end{figure}

Now we are ready to prove Proposition \ref{MainThm}.

\begin{proof}[Proof of Proposition \ref{MainThm}]

We divide the proof into three steps.

\textit{Step 1}: We consider the region $E<0$.  By the stable manifold theorem (see \cite{P}, for example), there is a $\delta>0$ such that  a smooth solution $(\rho^u,E^u)$ to \eqref{ODE_n_E} with $\rho^u>1$ and $E^u<0$ exists and $|(\rho^u-1,E^u)| \leq \delta$ for all $\xi \leq  0$. In particular, $(\rho^u,E^u) \to (1, 0)$ as $\xi \to -\infty$ exponentially fast.

We solve the initial value problem \eqref{ODE_n_E} with $(\rho,E)(0)=(\rho^u,E^u)(0)$. From \eqref{1st Int}, the (uniquely extended) solution $(\rho^u,E^u)$ satisfies 
\begin{equation}\label{09}
E^u = - \sqrt{2 (g(\rho^u)-g(1))} \quad \text{as long as} \quad E^u<0.
\end{equation} 

We show that there is a point $\xi^\ast\in(0,\infty)$ such that $\textstyle\lim_{\xi \nearrow \xi^\ast}E^u(\xi) = 0$. If not, there are two possible cases: (i) $\textstyle\lim_{\xi \nearrow \xi^M} E^u < 0$, where $\xi^M \in (0,\infty)$ is the maximal existence time, and (ii) $\textstyle\lim_{\xi \to +\infty} E^u \leq 0$. The case (i) is ruled out thanks to the local existence, \cite{P}. We claim that the case (ii) cannot occur.

Integrating \eqref{ODE_n_E1} on $[0,\xi]$, and then using \eqref{09},  we have
\begin{equation}\label{333}
\xi  = \int_{0}^\xi \frac{h(\rho^u(\tilde{\xi}))}{\sqrt{2 (g(\rho^u(\tilde{\xi})-g(1))}} \frac{d\rho^u}{d\tilde{\xi}}\,d{\tilde{\xi}} = \int_{\rho^u(0)}^{\rho^u(\xi)} \frac{h(\rho)}{\sqrt{2 (g(\rho)-g(1))} }\,d\rho.
\end{equation}
Here, the change of variable is possible since $(\rho^u)'>0$ by \eqref{Direction}.

We first consider the case $\kappa=0$. By expanding $g(\rho)$ at $\rho=\infty$, we obtain $g(\rho) -g(1) = c_0^2 \rho^{-1} + O(\rho^{-2})$ as $\rho \to +\infty$.  Hence, from \eqref{333},
\[
\xi   \leq \int_{\rho^u(0)}^\infty \frac{h(\rho)}{\sqrt{2 (g(\rho)-g(1))} }\,d\rho < \infty,
\]
which implies that as long as $(\rho^u)'>0$, or equivalently, $E^u<0$ by \eqref{Direction}, the case (ii) cannot happen. Hence, we conclude that there is a point $\xi^\ast$ such that  $\textstyle\lim_{\xi \nearrow \xi^\ast} E^u   = 0$. Moreover, from \eqref{09}, Lemma \ref{Lemma1} and Lemma \ref{Lemma2}, we also have that  $\textstyle\lim_{\xi \nearrow \xi^\ast} \rho^u   = +\infty$

To consider the case $\kappa>0$, we expand $h$ and $g$, and we obtain that  as $\rho \to \rho^\ast=c_\kappa/\sqrt{\kappa}$,
\begin{equation}\label{099}
\left\{
\begin{array}{l l}
h(\rho) =    \left( -\frac{3c_\kappa^2}{(\rho^\ast)^4} + \frac{\kappa}{(\rho^\ast)^2}  \right)(\rho-\rho^\ast) + O(|\rho-\rho^\ast|^2),  \\
g(\rho)  - g(1)  =   \frac{1}{2}\left( \frac{2c_\kappa^2}{(\rho^\ast)^3} +  e^{H(\rho^\ast)} \frac{dh}{d\rho}(\rho^\ast) \right)(\rho-\rho^\ast)^2 + O(|\rho-\rho^\ast|^3),
\end{array}
\right.
\end{equation} 
 where we have used the fact that $g(\rho^\ast)=g(1)$ and $h(\rho^\ast)=0$ (see Lemma \ref{Lemma1} and \eqref{sign h}). It is straightforward to check that the coefficients of the leading order terms of the right-hand sides of \eqref{099} are nonzero.  In a similar fashion to the case $\kappa=0$, one can show that there is a point $\xi^\ast\in (0,\infty)$ such that $\lim_{\xi \nearrow \xi^\ast} E^u(\xi)   = 0$ and $\lim_{\xi \nearrow \xi^\ast} \rho^u(\xi) = \rho^\ast=c_\kappa\sqrt{\kappa}$.
 
\textit{Step 2:} Now we consider the region $E>0$.  By the stable manifold theorem, there is a smooth solution $(\rho^s,E^s)$ to \eqref{ODE_n_E} with $\rho^s>1$ and $E^s>0$, and $(\rho^s,E^s)$ tends to $(1,0)$ as $\xi \to +\infty$ exponentially fast. By the same argument as the previous step, one can show that there is a point $\xi^{\ast\ast}$ such that $\textstyle \lim_{\xi \searrow \xi^{\ast\ast}} E^s(\xi) =0$ and $\textstyle \lim_{\xi \searrow \xi^{\ast\ast}} \rho^s(\xi) = \rho^\ast$. We omit the details.

\textit{Step 3:} We set $\xi^\ast = 0 = \xi^{\ast\ast}$ and define
\begin{equation*}
(\tilde{\rho},\tilde{E}):=
\left\{
\begin{array}{l l}
(\rho^u,E^u)(\xi) \quad \text{for } \xi<0, \\ 
(\rho^s,E^s)(\xi) \quad \text{for } \xi>0.
\end{array} 
\right.
\end{equation*} 
Then, $(\tilde{\rho},\tilde{E})$ satisfies \eqref{ODE_n_E} for $\xi \neq  0$ and \eqref{bdCon n E}. At $\xi = 0$, it is not differentiable. From the reduction discussed in subsection \ref{reduction}, we obtain  $(\rho,v,\phi)$ satisfying \eqref{TravelEq} at $\xi\neq 0$. Using the fact that the first integral $\Psi(\rho,E)$ defined in \eqref{1st Int} is symmetric about $\rho$ axis, one can check that $(\rho,v,\phi)$ satisfies \eqref{bdCon+-inf}--\eqref{Symmetric}. It is easy to check that \eqref{mono}--\eqref{peak} and \eqref{Maxk0}--\eqref{Maxk}. We complete the proof.

\end{proof}

\section{Asymptotic behaviors of the peaked solitary waves}

\subsection{Asymptotic behaviors at the peaks}  We investigate  asymptotic behaviors of the peaked solitary waves at the peaks for each fixed $\kappa \geq 0$.
\begin{proof}[Proof of Theorem \ref{Thm2}]
We first consider the case $\kappa=0$.   In this case,  $H^{-1}(\phi)$ is explicit:
\begin{equation}\label{11}
\rho = H^{-1}(\phi)=\frac{c_0}{\sqrt{c_0^2-2\phi}}.
\end{equation}
Hence, from \eqref{TravelEq3}, we obtain
\begin{equation}\label{111_0}
\phi'' = e^\phi - \frac{c_0}{\sqrt{c_0^2-2\phi}}.
\end{equation}
Multiplying \eqref{111_0} by $\phi'$, and then integrating it from $\xi$ to  $+\infty$, we have 
\begin{equation*}
\frac{(\phi')^2}{2} = e^\phi + c_0\sqrt{c_0^2-2\phi} - (1+c_0^2) \quad \text{for } \xi \in (0,+\infty).
\end{equation*}
Let $\Phi(\xi):=\phi^\ast-\phi(\xi)=c_0^2/2- \phi$. Since $\lim_{\xi \searrow 0}\Phi = 0 $, $\Phi>0$ and $\Phi' > 0$  for $\xi>0$, we get
\begin{equation}\label{111_3}
\begin{split}
\Phi'
& =\sqrt{2} \left( \exp\left(\frac{c_0^2}{2} - \Phi \right) + c_0\sqrt{2\Phi} - (1+c_0^2) \right)^{1/2}\\
& = \sqrt{2} \left(  (1+c_0^2)(e^{-\Phi} - 1) +c_0\sqrt{2\Phi} \right)^{1/2} \\
& = (2c_0)^{1/2}(2\Phi)^{1/4}( 1 + o(1) )  
\end{split}
\end{equation}
as $\xi \searrow 0$. Here, we have used \eqref{Eq z0 Cold-1} in the second equality.  Hence, for any small $\veps>0$, there exists $\xi_\veps<0$ such that for all $\xi\in(0,\xi_\veps)$,
\begin{equation}\label{111}
(2c_0)^{1/2}2^{1/4}(1-\veps) \leq \frac{4}{3} (\Phi^{3/4})' \leq  (2c_0)^{1/2}2^{1/4}(1+\veps).
\end{equation}
Integrating \eqref{111} from $0$ to $\xi$, and using the symmetry of $\Phi$, we have 
\begin{equation}\label{111_1}
\lim_{\xi \to 0}\frac{\Phi}{|\xi|^{4/3}} = \frac{1}{2}\left(\frac{3\sqrt{c_0}}{\sqrt{2}}\right)^{4/3}=:A_0.
\end{equation}
Combining \eqref{111_3} and \eqref{111_1}, we obtain 
\begin{equation*}\label{111_4}
\lim_{\xi \to 0}\frac{|\phi'|}{|\xi|^{1/3}} = (2c_0)^{1/2}2^{1/4}A_0^{1/4} = \frac{4}{3}A_0.
\end{equation*}
From \eqref{11} and \eqref{111_1}, it follows that
\begin{equation}\label{111_2}
\lim_{\xi \to 0}\rho|\xi|^{2/3} = \frac{c_0}{\sqrt{2A_0}} = \frac{4}{9}A_0.
\end{equation}
From \eqref{TravelEq3} and \eqref{111_2}, we get 
\begin{equation*}\label{111_5}
\lim_{\xi \to 0}-\phi''|\xi|^{2/3} = \lim_{\xi \to 0}(\rho-e^\phi)|\xi|^{2/3}= \frac{4}{9}A_0.
\end{equation*}
From \eqref{TravelEqB1} and \eqref{111_2}, we have
\begin{equation*}
\lim_{\xi \to 0}\frac{v^\ast - v}{|\xi|^{2/3}} = \lim_{\xi \to 0}\frac{c_0}{\rho|\xi|^{2/3}} = \sqrt{2A_0}.
\end{equation*}
This finishes the proof of \eqref{Thm2_1}.
 
Now, we consider the case $\kappa>0$. \eqref{Thm2_3} directly follows by using L'Hôpital's rule and  \eqref{TravelEq3} together with the fact that $\rho^\ast < \infty$ when $\kappa>0$.  
Using \eqref{Def_H} and expanding $H(\rho)$, we obtain
\[
\phi^\ast - \phi = -\frac{1}{2}\frac{dh}{d\rho}(\rho^\ast)(\rho-\rho^\ast)^2 \left( 1 + o(1) \right)
\]
as $\xi \to 0$.  Hence,
\begin{equation}\label{Thm2_5}
\lim_{\xi \to  0}\frac{\phi^\ast - \phi}{\xi^2} = -\frac{1}{2}\frac{dh}{d\rho}(\rho^\ast)\lim_{\xi \to 0}\frac{(\rho^\ast - \rho)^2}{\xi^2}.
\end{equation}
Using \eqref{TravelEqB1}, we get 
\begin{equation}\label{222_5}
\lim_{\xi \to  0}\frac{v^\ast - v}{|\xi|} = \frac{c_\kappa}{(\rho^\ast)^2} \lim_{\xi \to  0}\frac{\rho^\ast - \rho}{|\xi|}.
\end{equation}
Combining \eqref{Thm2_3}, \eqref{Thm2_5}, and \eqref{222_5}, we obtain \eqref{Thm2_4}. We finish the proof. 
\end{proof}

\subsection{Cold ion limit of peaked solitary waves}
In this subsection, we prove Theorem \ref{thm3}. 

Throughout this subsection, $\kappa_0$ appearing in each lemma and proposition denotes various positive and \textit{small} constants, and $C, N, \ldots$ (sometimes with indices) denote different constants depending only on the prescribed quantities, such as $\kappa_0, \cdots$.
This dependence is indicated in the parentheses: $N = N(\kappa_0, \ldots)$, $C = C(\kappa_0, \cdots)$.
For given two functions (or quantities) $f$ and $g$, we say $f \lesssim g$ if there is a constant $C$, independent of $\xi$ and $\kappa$, satisfying $f \le C g$.
We also denote $f \simeq g$ if $f \lesssim g$ and $g \lesssim f$. 

In what follows, we set  
\[
H_\kappa(\rho) := \frac{c_\kappa^2}{2}\left(1-\frac{1}{\rho^2} \right) - \kappa \ln\rho
\]
(see \eqref{Def_H}){, which is invertible on $[1,\rho_\ast^\kappa]$ from \eqref{sign h}.} Since $c_\kappa<c_0$ by Lemma \ref{z_K est}, we see that   $H_\kappa(\rho)<H_0(\rho)$ on the interval $(1,\infty)$, and thus  $H_\kappa^{-1}(\phi)>H_0^{-1}(\phi)$ on $(0,\phi_\kappa^\ast)$ (recall $\phi_\kappa^\ast   < \phi_0^\ast$). 
Also, as $\kappa \to 0$, we have 
\[
c_0^2/2 -o(1) = H_\kappa(\rho_\kappa^\ast) = \phi_\kappa^\ast  < c_0^2/2 = \phi_0^\ast 
\]
since $\rho_\kappa^\ast=c_\kappa/\sqrt{\kappa}$ and  $c_\kappa>\sqrt{\kappa}$.

Note that $\phi_\kappa$ with $\kappa \geq  0$ satisfies a second-order ODE with a nonlinear term $H_\kappa^{-1}$ in \eqref{12_0809}.
Observe that $H_\kappa^{-1}(\phi_\kappa)$ dominates $e^{\phi_\kappa}$ near $\xi=0$ due to the facts that $\rho_\kappa^\ast = H_\kappa^{-1}(\phi_\kappa^*) \to \infty$ as $\kappa\to 0$ and $\phi_\kappa(\xi)$ is close to $\phi_\kappa^*$ near $\xi=0$.
In order to analyze the shape of $\phi_\kappa$ in such a region, we first investigate the properties of $H_\kappa^{-1}$ in Lemmas \ref{H_K conv} and \ref{lem3.3}.
We refer to the auxiliary figure provided in Figure \ref{FigNume0} for Lemmas \ref{H_K conv} and \ref{lem3.3}.
\begin{lemma}
	\label{H_K conv}
For any $\veps \in (0, c_0^2/2)$,  $H_\kappa^{-1} \to H_0^{-1}$ in $C^1([0, c_0^2/2-\veps])$ as $\kappa\to 0$.
\end{lemma}
\begin{proof}
We set $I_\veps := [0, c_0^2/2-\veps]$.
First, by \eqref{Def_H} and \eqref{Def_h}, we note that $H_{\kappa}$ is increasing on $\rho \in [1,c_{\kappa}/\sqrt{\kappa}]$.
We easily see that, there are constants $\kappa_0 \ll 1$ and $M_0>1$ such that $H_{\kappa}^{-1}(I_\veps ) \subset [1,M_0]$
for all $\kappa \in [0, \kappa_0)$.
In particular, there exists a constant $C > 0$ independent of $ \kappa$ such that
\[
\frac{dH_{\kappa}^{-1}}{d\phi} = \frac{1}{ h_{\kappa} \left( H_{\kappa}^{-1} \right)}  < C \quad \text{ on }  \quad I_\veps
 \]
 since $h_\kappa = c_\kappa^2/ \rho^{3} - \kappa / \rho$ uniformly converges to $h_0 = c_0^2/\rho^3 > 0$ on $[1, M_0]$ as $\kappa \to 0$.
By the monotonicity $H_{\kappa} < H_{0}$ on $(1, \infty)$, it is clear that $H_{\kappa} \left(  H_0^{-1} \left( \phi \right)   \right) \in I_{\veps }$ for  $\phi \in I_{\veps}$.
Therefore, by the fact that $c_{\kappa} \to c_0$ as $\kappa \to 0$, we have
\[
\begin{split}
\left| H_0^{-1} \left( \phi \right) - H_{\kappa}^{-1} \left( \phi \right) \right| 
& =  \left|  H_{\kappa}^{-1}\left( H_{\kappa} \left(  H_0^{-1} \left( \phi \right)   \right)   \right)  - H_{\kappa}^{-1} \left( \phi \right)  \right| \\
& \le \|(H_{\kappa}^{-1})'\|_{L^\infty(I_{\veps})} |H_{\kappa}(H_0^{-1}(\phi))-  H_0 \left(  H_0^{-1} \left( \phi \right) \right) | \to 0 
\end{split}
\] 
uniformly on $I_\veps$, as $\kappa \to 0$.
By using the above, we also have
\[
\begin{split}
\left| \left( H_{\kappa}^{-1} \right)' \left( \phi \right) - \left( H_0^{-1} \right)' \left( \phi \right) \right| 
&  = \left|  \frac{1}{h_{\kappa} \left( H_{\kappa}^{-1}  \left( \phi \right) \right)}  
 -   \frac{1}{h_{0} \left( H_{0}^{-1}  \left( \phi \right) \right)}  \right| \\
 & \lesssim \left|  h_{0} \left( H_{0}^{-1}  \left( \phi \right) \right) - h_{0} \left( H_{\kappa}^{-1}  \left( \phi \right) \right) \right|  \\
 & \quad  + \left|  h_{0} \left( H_{\kappa}^{-1}  \left( \phi \right) \right) - h_{\kappa} \left( H_{\kappa}^{-1}  \left( \phi \right) \right) \right| \\
 & \to 0
\end{split}
\]
as $\kappa \to 0$ uniformly on $I_\veps$.
The lemma is proved.
\end{proof}

\begin{lemma}
	\label{lem3.3} 
For any fixed constant $M\geq 2$, there exist constants $\kappa_0>0$ and $N(\kappa_0,M)>1$ such that  the following hold: 
for all $\kappa \in(0,\kappa_0)$,
\begin{equation}
	\label{lem3.3eq1-1}
\frac{N^{-1}}{\sqrt{\phi_\kappa^\ast-\phi}}  \le H_\kappa^{-1}(\phi) \le 
\frac{N}{\sqrt{\phi_\kappa^\ast-\phi}}
 \quad \text{if} \quad \phi \in \left(0,H_\kappa \left( \frac{\rho_\kappa^\ast}{M} \right) \right);
\end{equation}
\begin{equation}
	\label{lem3.3eq1}
N^{-1} \rho_\kappa^\ast  \le H_\kappa^{-1}(\phi) \le N\rho_\kappa^\ast \quad \text{if} \quad \phi \in \left(H_\kappa\left( \frac{\rho_\kappa^\ast}{M} \right) ,H_\kappa \left( \rho_\kappa^\ast\right) \right).
\end{equation}
\end{lemma}

\begin{proof}
Since $c_{\kappa} \to c_0 > 0$, we have  $\rho_\kappa^\ast=c_\kappa/\sqrt{\kappa} \to \infty$ as $\kappa \to 0$. Note that, for $\rho \in (1,\rho_\kappa^\ast)$,
\begin{equation*}
	\label{eq240510_01}
H_\kappa(\rho_\kappa^\ast)-H_\kappa(\rho) =\frac{\kappa (\rho_\kappa^\ast)^2}{2} \left(\frac{1}{\rho^2}-\frac{1}{(\rho_\kappa^\ast)^2}\right) - \kappa\log \left( \frac{\rho_\kappa^\ast}{\rho} \right) \le \frac{\kappa (\rho_\kappa^\ast)^2}{2\rho^2}.
\end{equation*}
Then it holds that for $\rho \in (1,\rho_\kappa^\ast/M]$ (i.e., $\rho_\kappa^\ast/\rho \in [M, \rho_\kappa^\ast)$),
\begin{equation}
	\label{lem3.3eq2}
 \frac{\kappa (\rho_\kappa^\ast)^2}{4\rho^2} 
 \le \frac{\kappa (\rho_\kappa^\ast)^2}{2}\left( \frac{1}{\rho^2}-\frac{1}{M^2\rho^2}\right) -\kappa \log \left( \frac{\rho_\kappa^\ast}{\rho} \right)
   \le H_\kappa(\rho_\kappa^\ast)-H_\kappa(\rho)
    \le \frac{\kappa (\rho_\kappa^\ast)^2}{2\rho^2},
\end{equation} 
where we have used the fact that
\[
\frac{x^2}{4} - \frac{x^2}{2M^2} - \log x > 0 \quad \text{if}\quad  x\geq M
\]
for any $M\geq 2$. Since $H_\kappa$ is strictly increasing on $\rho \in (1,\rho_\kappa^\ast)$ and it maps the interval $\rho \in (1,\rho_\kappa^\ast/M)$ to $\phi =H_{\kappa}(\rho) \in (0,H_\kappa(\rho_\kappa^\ast/M))$, the inequality \eqref{lem3.3eq2} implies
\[
\frac{1}{2} \sqrt{\kappa}\rho_\kappa^\ast (H_\kappa(\rho_\kappa^\ast)-\phi)^{-\frac12}
 \le \rho= H_\kappa^{-1}(\phi)
  \le \frac{\sqrt{\kappa}}{\sqrt{2}} \rho_\kappa^\ast (H_\kappa(\rho_\kappa^\ast)-\phi)^{-\frac12}
\]
for $\phi \in (0,H_\kappa(\rho_\kappa^\ast/M))$.
Since $\sqrt{\kappa} \rho_\kappa^\ast =c_0 + O(\sqrt{\kappa})$ for small $\kappa>0$, we see that there exist $\kappa_0>0$ and $N = N(\kappa_0) > 1$ such that, if $\kappa \in (0, \kappa_0)$, then
\begin{equation}
	\label{lem3.3eq3}
N^{-1} (H_\kappa(\rho_\kappa^\ast)-\phi)^{-\frac12} \le \rho= H_\kappa^{-1}(\phi) \le N (H_\kappa(\rho_\kappa^\ast)-\phi)^{-\frac12}
\end{equation}
for $\phi \in (0,H_\kappa(\rho_\kappa^\ast/M))$.
On the other hand, for $\phi \in (H_\kappa(\rho_\kappa^\ast/M), H_\kappa(\rho_\kappa^\ast))$, we have
\begin{equation}
	\label{lem3.3eq4}
\frac{\rho_\kappa^\ast}{M} \le H_\kappa^{-1}(\phi) \le \rho_\kappa^\ast,
\end{equation}
since $H_\kappa$ is bijective on $(1,\rho_\kappa^\ast)$.
By combining \eqref{lem3.3eq3} and \eqref{lem3.3eq4}, we obtain the result.
\end{proof}

\begin{figure}[h]
\begin{center}
\includegraphics[width=0.6\linewidth]{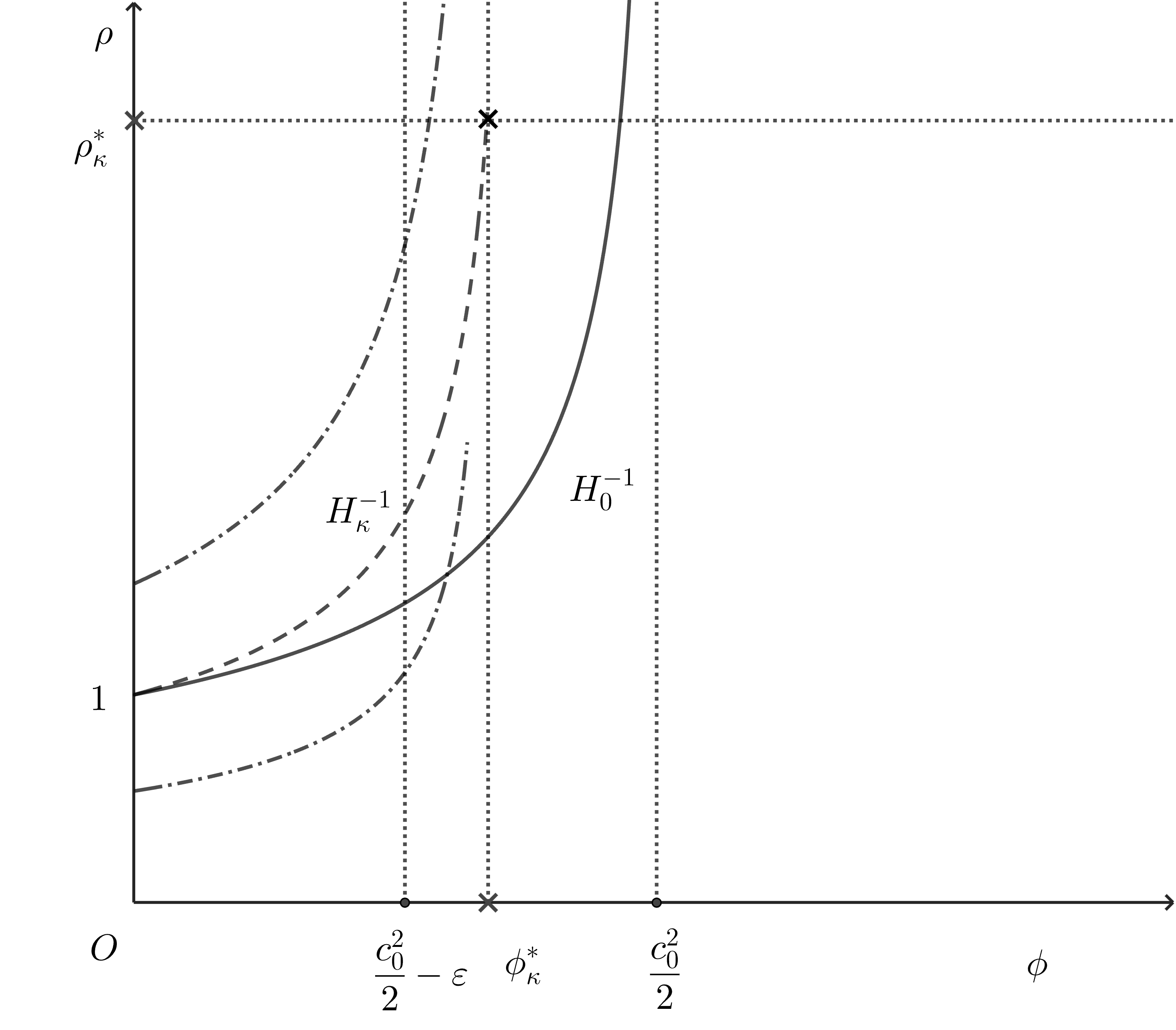}
\end{center}
\caption{An auxiliary figure for Lemmas \ref{H_K conv} and \ref{lem3.3}. Solid: $H_0^{-1}(\phi)$, dashed: $H_\kappa^{-1}(\phi)$ with $\kappa>0$, dash-dotted: positive functions $C(\phi_\kappa^\ast - \phi)^{-1/2}$ for some $C>0$. }
\label{FigNume0}
\end{figure}

\subsubsection{Behaviors near transition regime}
	\label{subsub_1}
In what follows, we investigate the shape of $\phi_\kappa(\xi)$ near $\xi=0$.
In this region, we analyze the shape of $\phi_\kappa$ by the properties of $H_\kappa^{-1}$ in Lemmas \ref{H_K conv} and \ref{lem3.3}. In particular, we divide the region into two parts  according to the behavior of $H_\kappa^{-1}$, which makes $\phi_\kappa$ behaves differently in each subregion; See \eqref{lem3.3eq1-1} and \eqref{lem3.3eq1}.\\

We first introduce some preliminaries. Let us denote
\begin{equation}	
	\label{eq240803_112}
\psi_\kappa(\xi):=H_\kappa(\rho_\kappa^\ast)-\phi_\kappa(\xi)=\phi_\kappa^*-\phi_\kappa(\xi).
\end{equation}
Note that $\psi_\kappa$ satisfies 
\begin{equation}\label{12_0809} 
\psi_\kappa'' = -\phi_{\kappa}'' = H_\kappa^{-1}(\phi_\kappa)-e^{\phi_\kappa}.
\end{equation}
We fix a constant $M \geq 2$. For $\kappa_0$ and $N(\kappa_0,M)$, that are positive constants from Lemma \ref{lem3.3}, we let $\veps_0>0$ be a constant such that 
\[
\veps_0 \le (2Ne^{\phi_\kappa^*})^{-2} \quad \text{for all } \kappa \in (0,\kappa_0).
\] 
For such $\veps_0>0$ and $\kappa \in (0,\kappa_0)$, we define intervals 
\begin{align}
	\label{eq240803_11}
I_\kappa & :=\left\{\xi \in \mathbb{R}:\psi_\kappa \left( \xi \right) \le H_\kappa \left( \rho_\kappa^\ast \right) - H_\kappa \left(\rho_\kappa^\ast/M \right) \right\}  \\
& = \left\{ \xi \in \mathbb{R}:     H_\kappa \left(\rho_\kappa^\ast/M \right) \le       \phi_\kappa(\xi)                \right\} \nonumber
\end{align}
and 
\begin{equation}
	\label{eq240529_01}
J_\kappa :=\{\xi \in \mathbb{R}:\psi_\kappa(\xi) \le \veps_0\}.
\end{equation}

We claim that
\begin{equation}
	\label{eq240510_02}
\frac{H_\kappa^{-1}(\phi_\kappa(\xi))}{2} \le \psi_\kappa''(\xi)  \le H_\kappa^{-1}(\phi_\kappa(\xi)) \quad \text{for} \quad \xi \in J_\kappa.
\end{equation}
The second inequality in \eqref{eq240510_02} directly follows from \eqref{12_0809}  since $e^{\phi_\kappa}>0$. To show the first inequality in \eqref{eq240510_02}, we use Lemma \ref{lem3.3}, and there are two cases: for $\xi \in J_\kappa$, $\phi(\xi)$  satisfies \eqref{lem3.3eq1-1} or \eqref{lem3.3eq1}.
If \eqref{lem3.3eq1} holds, then using the fact that $c_\kappa \to c_0 > 0$ as $\kappa \to 0$, there exists $\kappa_0 > 0$ such that for any $\kappa \in (0, \kappa_0)$,
\[
e^{\phi_\kappa(\xi)}
 \le  e^{\phi_\kappa^*} =  e^{\phi_\kappa^*} \frac{c_\kappa}{\sqrt{\kappa}} \frac{\sqrt{\kappa}}{c_\kappa }
 \le \frac12 H_\kappa^{-1}(\phi_\kappa(\xi)).
\]
If \eqref{lem3.3eq1-1} holds, then by \eqref{eq240529_01}, we have
\[
e^{\phi_\kappa(\xi)} \le  e^{\phi_\kappa^*} \le
 e^{\phi_\kappa^*} \veps_0^{1/2} \left( \psi_\kappa \left( \xi \right) \right)^{-1/2} \le  \frac{1}{2N} (\phi_\kappa^*-\phi_\kappa(\xi))^{-1/2}
 \le \frac12 H_\kappa^{-1}(\phi_\kappa(\xi))
\]
provided that $\veps_0 \le (2Ne^{\phi_\kappa^*})^{-2}$ ($N$ from Lemma \ref{lem3.3}). We finish the proof of the claim \eqref{eq240510_02}.\\

In the following two propositions, we obtain pointwise estimates of $\psi_\kappa(\xi)$ on $I_\kappa$ and $J_\kappa$, respectively. Also, we estimate the length of $I_\kappa$ and $J_\kappa$ with respect to the parameter $\kappa$.
In particular, we will see that $I_\kappa \subsetneq J_\kappa$ for sufficiently small $\kappa>0$.

We first deal with the case $\xi \in I_\kappa$.
\begin{proposition}
	\label{prop3.4}  
Let $\psi_\kappa$ and $I_\kappa$ be as in \eqref{eq240803_112} and \eqref{eq240803_11}, respectively.
Then there exist $\kappa_0 > 0$ and $C = C(\kappa_0) > 1$ such that $\psi_\kappa$ and $I_\kappa$ satisfy the following: 
for all $\kappa \in (0, \kappa_0)$, we have
\begin{equation*}
  C^{-1} \rho_\kappa^\ast\xi^2 \le \psi_\kappa(\xi) \le C \rho_\kappa^\ast \xi^2, \quad   C^{-1} \rho_\kappa^\ast \le \psi_\kappa''(\xi) \le C \rho_\kappa^\ast \quad \text{on} \quad I_\kappa,
\end{equation*}
where $\rho_\kappa^\ast \simeq \kappa^{-1/2}$.
Moreover, the length $|I_\kappa|$ of $I_\kappa$ satisfies
$|I_\kappa| \simeq \kappa^{3/4}$. Specifically, 
\begin{equation}\label{I1}
I_\kappa = [-N_\kappa \kappa^{3/4}, N_\kappa \kappa^{3/4}]
\end{equation}
for some $N_\kappa \simeq 1$.
\end{proposition}

\begin{proof}
We recall $\rho_\kappa^\ast = c_{\kappa}/\sqrt{\kappa}$. 
Note that $I_\kappa \subsetneq J_\kappa$ for sufficiently small $\kappa>0$ since
\begin{equation}	
	\label{eq240821_01}
H_\kappa(\rho_\kappa^\ast)-H_\kappa \left( \frac{\rho_\kappa^\ast}{M} \right)=\frac{\kappa}{2} \left( M^2 - 2\log M -1\right) < \veps_0.
\end{equation}

Since $H_\kappa(\rho_\kappa^* / M) \le \phi_\kappa(\xi)$ on $I_\kappa$, we have that
\[
 \rho_\kappa^\ast \simeq H_\kappa^{-1}\left( \phi_\kappa  \left( \xi \right) \right) \simeq \psi_\kappa''(\xi)  \quad \text{on} \quad I_\kappa \subset J_\kappa
\]
by \eqref{lem3.3eq1} of Lemma \ref{lem3.3} and \eqref{eq240510_02}, respectively.
This  with the fact that $\psi_\kappa(0) = \psi_\kappa'(0) = 0$ and $\sqrt{\kappa} \rho_\kappa^\ast = c_\kappa \to c_0 > 0$ as $\kappa \to 0$  implies that
\begin{equation}
	\label{psi1}
\psi_\kappa'(\xi) \simeq \rho_\kappa^\ast \xi \simeq \kappa^{-1/2}\xi \quad \text{and} \quad \psi_\kappa(\xi) \simeq \rho_\kappa^\ast \xi^2 \simeq \kappa^{-1/2}\xi^2 \quad \text{for}\quad  \xi \in I_\kappa \subset J_\kappa.
\end{equation}
Observe that \eqref{psi1} with the definition \eqref{eq240803_11} of $I_\kappa$ and  the equality in \eqref{eq240821_01} implies 
\[
\{ \xi \in \mathbb{R}: | \xi | \le L_\kappa \kappa^{3/4} \} \subset I_\kappa \subset \{ \xi \in \mathbb{R}: | \xi | \le R_\kappa \kappa^{3/4} \}
\]
for some $L_\kappa, R_\kappa \simeq 1$ (for sufficiently small $\kappa$). Hence, \eqref{I1} holds true.
The proposition is proved.
\end{proof}

 Now we  deal with the case $\xi \in J_\kappa \setminus I_\kappa$.
\begin{proposition}
	\label{prop3.4_1}  
Let $\psi_\kappa$ and $J_\kappa$ be as in \eqref{eq240803_112} and \eqref{eq240529_01}, respectively.
Then there exist $\kappa_0 > 0$ and $C = C(\kappa_0) > 1$ such that $\psi_\kappa$ and $J_\kappa$ satisfy the following: for all $\kappa \in (0, \kappa_0)$, we have
 \begin{equation} 
	\label{eq240716_01}
C^{-1} \le \, \frac{\psi_\kappa(\xi)}{\xi^{4/3}}, \, \frac{\psi_\kappa'(\xi)}{\xi^{1/3}}, \,  \frac{\psi_\kappa''}{\psi_\kappa^{-1/2}} \, \le C.
\end{equation}
 for $\xi \in J_\kappa \setminus I_\kappa$ with $\xi>0$. 
 Moreover, the length $|J_\kappa|$ of $J_\kappa$ is  of order $1$, i.e., $ |J_\kappa|  \simeq 1$.
\end{proposition}

\begin{proof}
Since $I_\kappa = \{\xi \in \mathbb{R}: H_\kappa(\rho_\kappa^*/M) \le \phi_\kappa(\xi) \}$, it holds that $\phi_\kappa(\xi) < H_\kappa(\rho_\kappa^*/M)$ for $\xi \in J_\kappa \setminus I_\kappa$.
Then by \eqref{lem3.3eq1-1} of Lemma \ref{lem3.3} and \eqref{eq240510_02}, we have
\[
\psi_\kappa^{-1/2}(\xi) = \frac{1}{\sqrt{\phi_\kappa^* - \phi(\xi) }} \simeq H_\kappa^{-1}\left( \phi  \left( \xi \right) \right)  \simeq  \psi_\kappa''(\xi)
\]
for $\xi \in J_\kappa \setminus I_\kappa$.
In other words, there are constants $a, A>0$ such that 
\begin{equation}\label{psi2}
a \psi_\kappa^{-1/2} \le \psi_\kappa'' \le A \psi_\kappa^{-1/2} \quad \text{for} \quad \xi \in J_\kappa \setminus I_\kappa,
\end{equation}
which proves \eqref{eq240716_01} involving $\psi_\kappa''$.
\\

We prove $\psi_\kappa \simeq \xi^{4/3}$ and $\psi_\kappa' \simeq \xi^{1/3}$ using the relation \eqref{psi2}.
First, by multiplying $\psi_\kappa'$ to the first inequality of \eqref{psi2}, and then integrating it on $(N_\kappa \kappa^{3/4}, \xi)$ for $\xi \in (J_\kappa \setminus I_\kappa) \cap \mathbb{R}_+$ (recall $\psi_\kappa' > 0$ for $\xi \in \mathbb{R}_+$), we have
\[
2a\left((\psi_\kappa)^{1/2}(\xi)-(\psi_\kappa)^{1/2}(N_\kappa \kappa^{3/4})\right) \le \frac12 \left((\psi_\kappa')^2(\xi)-(\psi_\kappa')^2(N_\kappa \kappa^{3/4})\right) .
\]
On the other hand, from \eqref{I1} and \eqref{psi1}, we see that
\begin{equation}\label{eq0824B1}
\frac{(\psi_\kappa')^2(N_\kappa \kappa^{3/4})}{(\psi_\kappa)^{1/2}(N_\kappa \kappa^{3/4})} \simeq \frac{(\rho_\kappa^\ast)^2 N_\kappa^2 \kappa^{3/2}}{(\rho_\kappa^\ast)^{1/2} N_\kappa \kappa^{3/4}} = N_\kappa (\rho_\kappa^\ast)^{3/2}  \kappa^{3/4} \simeq c_\kappa^{3/2} \to c_0^{3/2} > 1 
\end{equation} 
as $\kappa \to 0$. 
Thus, there are positive constants $\kappa_0 $ and $b = b(\kappa_0) <1$ such that, for all $\kappa \in (0, \kappa_0)$, we have
\begin{equation}\label{phi3}
\begin{split}
(\psi_\kappa')^2(\xi)
& \ge 4a(\psi_\kappa)^{1/2}(\xi) + (\psi_\kappa')^2(N_\kappa \kappa^{3/4})-4a(\psi_\kappa)^{1/2}(N_\kappa \kappa^{3/4}) \\
& \ge 4a (\psi_\kappa)^{1/2}(\xi) - 4a(1-b)(\psi_\kappa)^{1/2}(N_\kappa \kappa^{3/4}).
\end{split}
\end{equation}
By using the fact that $\psi_\kappa$ is strictly increasing on $\mathbb{R}_+$, we obtain from \eqref{phi3} that 
\[
(\psi_\kappa')^2(\xi) \geq 4ab (\psi_\kappa)^{1/2}(\xi)
\]
if we choose sufficiently small $b > 0$ uniformly on $\mathbb{R}_+ \cap \overline{(J_\kappa\setminus I_\kappa)}$.
Therefore, for some constant $\tilde{b} >0$,   we have
\begin{equation}
	\label{eq240511_01}
(\psi_\kappa^{3/4})'(\xi) = \frac34 (\psi_\kappa)'(\xi) (\psi_\kappa)^{-1/4}(\xi)  \geq \tilde{b} \quad \textrm{for} \quad \xi \in \mathbb{R}_+ \cap (J_\kappa\setminus I_\kappa).
\end{equation}
Combining \eqref{I1} and \eqref{psi1} with the facts that $N_{\kappa} \simeq 1$ and  $\sqrt{\kappa}\rho_\kappa^\ast = c_{\kappa} \to c_0 > 0$ as $\kappa \to 0$,   we can choose sufficiently small $\kappa_0$ such that for all $\kappa \in (0,\kappa_0)$, it holds that 
\[
\psi_{\kappa}(N_{\kappa} \kappa^{3/4}) \simeq \rho_\kappa^\ast \kappa^{3/2} \simeq  \kappa.
\]
Hence, there exists a constant $\hat{b}(\kappa_0)>0$  such that for all $\kappa \in (0,\kappa_0)$, we have
\[
\psi_\kappa^{3/4}(N_\kappa \kappa^{3/4})  \geq  \int_0^{N_{\kappa}\kappa^{3/4}} \hat{b} \,d\xi.
\]  
Combined with \eqref{eq240511_01}, this implies that 
\begin{equation}\label{psi4}
\psi_\kappa^{3/4}(\xi) \ge \int_{N_\kappa \kappa^{3/2}}^\xi \tilde{b}\,d\xi + \psi_\kappa^{3/4}(N_\kappa \kappa^{3/4}) \geq  \int_0^{\xi} \min \{ \tilde{b}, \hat{b} \} \,d\xi \gtrsim \xi
\end{equation}
for $\xi \in \mathbb{R}_+ \cap (J_\kappa\setminus I_\kappa)$ and $\kappa \in (0, \kappa_0)$.
Together with the relation \eqref{eq240511_01}, we obtain from  \eqref{psi4} that 
\begin{equation}
	\label{eq240821_02}
\psi_\kappa (\xi) \gtrsim \xi^{4/3} \quad \text{and} \quad \psi_\kappa'(\xi) \gtrsim \xi^{1/3} \quad \text{for} \quad \xi \in  \mathbb{R}_+ \cap (J_\kappa\setminus I_\kappa).
\end{equation}

We proceed in a similar way to obtain inequalities \eqref{eq240821_02} with $\gtrsim$ replaced by $\lesssim$.  
From \eqref{psi2}, we have
\[
\frac12 \left((\psi_\kappa')^2(\xi)-(\psi_\kappa')^2(N_\kappa \kappa^{3/4})\right) \le 2A\left((\psi_\kappa)^{1/2}(\xi)-(\psi_\kappa)^{1/2}(N_\kappa \kappa^{3/4})\right).
\]
Using \eqref{eq0824B1} and that $\psi_\kappa$ is increasing on $\mathbb{R}_+$, we see  
\begin{align*}
(\psi_\kappa')^2(\xi)& \le 4A(\psi_\kappa)^{1/2}(\xi) + (\psi_\kappa')^2(N_\kappa \kappa^{3/4})-4A(\psi_\kappa)^{1/2}(N_\kappa \kappa^{3/4}) \\
& \le 4A (\psi_\kappa)^{1/2}(\xi) + B(\psi_\kappa)^{1/2}(N_\kappa \kappa^{3/4}) \\
& \le (4A+B) (\psi_\kappa)^{1/2}(\xi)
\end{align*}
for some constant $B>0 $ independent of $\kappa \in (0, \kappa_0)$.
Therefore, $(\psi_\kappa)'(\xi) (\psi_\kappa)^{-1/4}(\xi) \le (4A+B)$ for $ \xi \in  \mathbb{R}_+ \cap (J_\kappa\setminus I_\kappa)$.
From this, we obtain that
\begin{equation}\label{psi5}
\psi_\kappa (\xi) \lesssim \xi^{4/3} \quad \text{and} \quad \psi_\kappa'(\xi) \lesssim \xi^{1/3} \quad \text{for} \quad \xi \in  \mathbb{R}_+ \cap (J_\kappa\setminus I_\kappa).
\end{equation} \\
By combining \eqref{eq240821_02} and \eqref{psi5} we deduce that $\psi_\kappa (\xi) \simeq \xi^{4/3}$ and $\psi_\kappa' (\xi) \simeq \xi^{1/3}$ for $\xi \in J_\kappa \setminus I_\kappa$. 
This with \eqref{psi2} proves \eqref{eq240716_01}. Finally, we see that $|J_\kappa| \simeq 1$ by the fact in \eqref{eq240716_01} that $\psi_\kappa \simeq \xi^{4/3}$ on $(J_\kappa \setminus I_\kappa ) \cap \mathbb{R}_+$ and the definition of $J_\kappa = \{ \xi \in \mathbb{R} : \psi_\kappa(\xi) \le \varepsilon_0 \}$. The proposition is proved.
\end{proof}

\subsubsection{Uniform exponential decay and proof of Theorem \ref{thm3}}
	\label{subsub_2}

In order to estimate $\phi_\kappa$ in $\mathbb{R}\setminus J_\kappa =\{\xi \in \mathbb{R}: \psi_\kappa = \phi_\kappa^\ast - \phi_\kappa \geq \veps_0\}$, we analyze the map $\phi \mapsto e^{\phi}-H_\kappa^{-1}(\phi)$ for $\phi \in (0,c_0^2/2-\veps_0)$. (Here we note that $c_0^2/2-\veps_0 >  \phi_\kappa^\ast-\veps_0$.)
Although $H_\kappa^{-1}$ is not written explicitly, using Lemma \ref{H_K conv}, we obtain the following properties.

\begin{lemma}
	\label{lem3.5}
There exists $\kappa_0 > 0$ such that the following hold:
\begin{enumerate}
\item For any $\kappa \in [0, \kappa_0)$, there exists $a_\kappa>0$ such that 
\begin{equation}
	\label{step3eq1}
e^{\phi}-H_\kappa^{-1}(\phi)
\begin{cases}
 \ge 0 \quad \text{ for } \quad \phi \le a_\kappa,\\
\le 0  \quad \text{ for } \quad  \phi \ge a_\kappa.
\end{cases} 
\end{equation}
Moreover, $a_\kappa \to a_0$ as $\kappa \to 0$.

\item For any $\delta>0$, there exists $c=c(\delta)>0$ and small number $\kappa_\delta < \kappa_0$ such that, for $\kappa \in [0, \kappa_\delta)$,
\begin{equation}
	\label{step3eq4}
e^{\phi}-H_\kappa^{-1}(\phi) \ge c\phi \quad \text{ for } \quad  \phi \le a_\kappa-\delta.
\end{equation}

\end{enumerate}

\end{lemma}

\begin{proof}
\textit{(1)} The first inequality \eqref{step3eq1} can be obtained by taking $a_\kappa=H_\kappa(\hat{\rho}_{\kappa})$ for $\kappa \ge 0$ where $\hat{\rho}_{\kappa}$ from \eqref{Aux 6}.
Moreover, by \eqref{step3eq1} and uniform convergence of $H_\kappa^{-1}$ (Lemma \ref{H_K conv} ), we have that $e^\phi-H_0^{-1}(\phi) = 0$ on $[\liminf_{\kappa \to 0}a_\kappa,\limsup_{\kappa \to 0} a_\kappa]$.
This means that $a_\kappa$ converge to $a_0$ as $\kappa \to 0$ by the uniqueness of the zero of $e^\phi-H_0^{-1}(\phi)$ \eqref{Aux 6}.

\textit{(2)} On the other hand, one verify by direct calculations that, for given $\delta>0$, there exists $\tilde{c}=\tilde{c}(\delta)>0$ satisfying
\begin{equation*}
	\label{step3eq3}
e^{\phi}-H_0^{-1}(\phi) \ge \tilde{c} \phi \quad \text{ for } \quad  \phi \le a_0 - \frac{\delta}{2}.
\end{equation*}
Thus, by Lemma \ref{H_K conv} and the fact $a_\kappa \to a_0$, it holds that for sufficiently small $\kappa>0$,
\begin{equation*}
e^{\phi}-H_\kappa^{-1}(\phi) \ge \frac{\tilde{c}}{2} \phi \quad \text{ for } \quad \phi \le a_\kappa- \delta.
\end{equation*}
The lemma is proved.
\end{proof}

With the help of Lemma \ref{lem3.5}, we obtain an uniform (in $\kappa$) exponential decay of $\phi_\kappa$ and $\rho_\kappa$.
\begin{proposition}\label{prop3.6} 
There exist constants $\kappa_0 > 0$ and $C, N > 0$, independent of $\kappa$, such that for all $\kappa \in (0,\kappa_0)$, 
\[\phi_\kappa(\xi), \,\, \rho_{\kappa}(\xi) - 1 \le C\exp(-N |\xi|) \quad \text{ for } \,\, \xi \in \mathbb{R} \setminus J_\kappa.\]
\end{proposition}

\begin{proof}
Owing to the symmetry, we only consider $\xi \in \mathbb{R}_+ \setminus J_\kappa$.  
 For each $\kappa \geq 0$, we set $\xi_\kappa:=\phi_\kappa^{-1}(a_\kappa)>0$, where $a_\kappa=H(\hat{\rho}_\kappa)$ from Lemma \ref{lem3.5}. Then, since $H_\kappa(\hat{\rho}_\kappa) = \phi_\kappa(\xi_\kappa) = H_\kappa(\rho_\kappa(\xi_\kappa))$ and $H_\kappa(\rho_\kappa(\xi))$ is strictly decreasing on $\xi \in \mathbb{R}_+$, we have $\hat{\rho}_\kappa = \rho_\kappa(\xi_\kappa)$. 

From \eqref{Aux 4}, we obtain
\[
\begin{split}
\frac{1}{2}(\phi_\kappa')^2(\xi_\kappa) 
& = \frac{c_\kappa^2}{\hat{\rho}_\kappa} + \kappa\hat{\rho}_\kappa + e^{H_\kappa(\hat{\rho}_\kappa)} - (c_\kappa^2 + \kappa + 1) \\
& = \frac{c_\kappa^2}{\hat{\rho}_\kappa} + (\kappa+1)\hat{\rho}_\kappa - (c_\kappa^2 + \kappa + 1),
\end{split}
\]
where we have used the fact that $\hat{\rho}_\kappa = e^{H_\kappa(\hat{\rho}_\kappa)}$ (see subsection \ref{Stationary}). Since 
\[
\lim_{\kappa \to 0} \left( \frac{c_\kappa^2}{\hat{\rho}_\kappa} + (\kappa+1)\hat{\rho}_\kappa - (c_\kappa^2 + \kappa + 1) \right) =  \frac{c_0^2}{\hat{\rho}_0} + \hat{\rho}_0 - (c_0^2 + 1) =\frac12 (\phi_0')^2(\xi_0)  > 0,
\]
we see that there exist positive constants $\tilde{\kappa}_0$ and $\mu(\tilde{\kappa}_0)$ such that for all $\kappa\in[0, \tilde{\kappa}_0)$,
\begin{equation}\label{Au3}
4\mu>-\phi_\kappa'(\xi_\kappa)>2\mu.
\end{equation} 

We claim that 
\[
0<\xi_\ast:=\inf_{\kappa\in [0, \tilde{\kappa}_0)}\xi_\kappa \leq  \sup_{\kappa\in [0, \tilde{\kappa}_0 )}\xi_\kappa =:\xi^{\ast} <\infty.
\]
First of all, thanks to \eqref{step3eq1},  we observe that $\phi_\kappa$ is concave (i.e., $\phi_\kappa''<0$) on $(0,\xi_\kappa)$ and convex (i.e., $\phi_\kappa''>0$) on $(\xi_\kappa,\infty)$ since $\phi_\kappa'' = e^{\phi_\kappa} - H_\kappa^{-1}(\phi_\kappa)$ (see also Figure \ref{FigNumeric} in Section 2.4). Hence,
\[
\frac{\phi_\kappa(\xi_\kappa) - \phi_\kappa^\ast}{\xi_\kappa} = \phi'_\kappa(\tilde{\xi}_\kappa) \geq \phi'_\kappa(\xi_\kappa),
\] 
where $0<\tilde{\xi}_\kappa < \xi_\kappa$, and we deduce that  $0 < \xi_\ast$ using \eqref{Au3} and the fact that $\phi_\kappa(\xi_\kappa)-\phi_\kappa^\ast \to a_0-\phi_0^\ast<0$ as $\kappa \to 0$. On the other hand, since $\xi_\kappa>\xi_\ast>0$, the concavity of $\phi$ yields that  
\begin{equation}\label{Au1.0}
-\frac{\phi_\kappa^*}{\xi_\kappa- \xi_\ast} \le \frac{\phi_\kappa(\xi_\kappa)-\phi_\kappa(\xi_\ast)}{\xi_\kappa-\xi_\ast} \le \phi_\kappa'(\xi_\ast) \le \sup_{\kappa \in [0, \tilde{\kappa}_0 )}\phi_\kappa'(\xi_\ast) < 0,
\end{equation}
where we have used that $\phi_\kappa^\ast \geq \phi_\kappa(\xi_\ast) \geq \phi_\kappa(\xi_\ast) - \phi_\kappa(\xi_\kappa)$ for the first inequality. Now \eqref{Au1.0} implies that there is a  constant $\tilde{C}>0$, uniform in $\kappa \geq 0$, such that  $\xi_\kappa  \leq  \xi_\ast + C \phi_\kappa^\ast \leq \tilde{C}$ for all $\kappa\in[0, \tilde{\kappa}_0)$. Thus,  we have $\xi^\ast < \infty$.

We note that for $\xi > \xi_\kappa$,
\[
\left| \phi_\kappa'(\xi) \right| \geq  |\phi_\kappa'(\xi_\kappa)|  -  \int_{\xi_\kappa}^{\xi} \left| \phi_\kappa''(\tau) \right| \,d\tau \geq  2\mu -  \sup_{\tau \geq \xi^\ast}|\phi_\kappa''(\tau)| \cdot \left(\xi-\xi^\ast \right),
\] 
where we have used \eqref{Au3} in the last inequality.
Recalling the definition of $J_\kappa$ and that $a_\kappa \to a_0$ as $\kappa \to 0$, we may take a small $\veps_0$ such that $J_\kappa \subset (-\xi_\kappa,\xi_\kappa)$ since $\phi_\kappa$ strictly decreases on $(0,\infty)$. 
By Lemma \ref{H_K conv},  $\|\phi_\kappa''\|_{L^\infty(\mathbb{R} \setminus J_\kappa)}=\|e^{\phi_\kappa}-H_\kappa^{-1}(\phi_\kappa)\|_{L^\infty(\mathbb{R} \setminus J_\kappa)}$ is uniformly bounded for $\kappa \in (0, \tilde{\kappa}_0)$.
Thus, there exist $\mu = \mu(\tilde{\kappa}_0), \delta = \delta(\tilde{\kappa}_0)>0$ independent of $\kappa \in (0, \tilde{\kappa}_0)$ such that $-\phi_\kappa'(\xi) \ge \mu $ for $\xi \in [\xi_\kappa,\xi_\kappa+\delta]$.
Then since $\phi_\kappa$ is monotone decreasing,
\[
\mu\le \frac{\phi_\kappa(\xi_\kappa) - \phi_\kappa(\xi_\kappa + \delta)}{\delta} \le 
\frac{\phi_\kappa(\xi_\kappa) - \phi_\kappa(\xi)}{\delta}
\]
for $\xi \ge \xi_\kappa + \delta$, that is, 
\begin{equation}\label{0907b}
\phi_\kappa(\xi) \le \phi_\kappa(\xi_\kappa)- \mu  \delta \quad \text{ for } \xi \ge \xi_\kappa +\delta.
\end{equation}
Applying \eqref{step3eq4} of Lemma \ref{lem3.5}-(2) with $\phi_\kappa'' = e^{\phi_\kappa} - H_\kappa^{-1}(\phi_\kappa)$ and $\phi_\kappa(\xi_\kappa) = a_\kappa$, it follows from \eqref{0907b} that there are constants $\kappa_{\mu\delta}<\tilde{\kappa}_0$ and $c=c(\mu\delta)>0$ such that if $\kappa \in (0, \kappa_{\mu\delta})$, then
\begin{equation}
	\label{eq240731_2}
\phi_\kappa''(\xi) \ge c \phi_\kappa(\xi)
\end{equation}
 for $\xi \ge \xi_\kappa+\delta$.
Now we set $\kappa_0=\kappa_{\mu\delta}$ abusing notation.

By multiplying $\phi_\kappa'(\xi) (\le 0)$ on both sides of \eqref{eq240731_2} and then integrating from $\xi (\ge \xi_\kappa + \delta)$ to $\infty$ (recall that $\phi_\kappa(\infty) = \phi_\kappa'(\infty) = \phi_\kappa''(\infty) = 0$), we deduce
\[
\sqrt{c} \le \frac{-\phi_\kappa'(\xi)}{\phi_\kappa(\xi)} = - \frac{d}{d\xi} \log \left( \phi_\kappa \left( \xi \right) \right) \quad \text{ for } \,\, \xi \ge \xi_\kappa +\delta,
\]
and this implies
\[
\phi_\kappa(\xi) \le C\exp(-\sqrt{c}(\xi-\xi_\kappa)) \quad \text{ for } \,\, \xi \ge \xi_\kappa +\delta.
\]
Thus we deduce the  uniform exponential decay of $\phi_\kappa$ on $\mathbb{R}\setminus J_\kappa$.
 
On the other hand, from the definition \eqref{Def_H}, one can find positive constants $\eta$ and $\tilde{\eta}$, independent of $\kappa \in (0, \kappa_0)$, such that
\[
\eta(\rho -1) < H_\kappa(\rho) - H_\kappa(1) = H_\kappa(\rho) \quad \text{ for } \,\,\rho \in [1, 1+\tilde{\eta}].
\]
This and the exponential decay of $\phi_\kappa$ imply that there are positive constants $R_0, C$ and $N$ (independent of small $\kappa$) satisfying the following:
\[
\rho_\kappa - 1 = H_{\kappa}^{-1}(\phi_{\kappa}) - 1 \le \eta^{-1}H_{\kappa}\left( H_{\kappa}^{-1} \left( \phi_\kappa \right) \right) = \eta^{-1}\phi_{\kappa} \le C\exp (-N|\xi|)
\]
for $|\xi| > R_0 $. This with the uniform boundedness of $\rho_\kappa$ on $\mathbb{R} \setminus J_\kappa$ (see Lemma \ref{H_K conv}), we obtain the desired result.
The proposition is proved.
\end{proof}

Now we prove Theorem \ref{thm3}.
\begin{proof}[Proof of Theorem \ref{thm3}]
($1$) We first prove the convergence \eqref{Eq_Thm3} for $0<\alpha<1/3$. By Proposition \ref{prop3.4}, it holds that 
\begin{subequations}\label{eq13_0809}
\begin{align} 
& | \phi_\kappa'' (\xi) | = |\psi_\kappa'' (\xi)| \simeq \frac{c_\kappa}{\sqrt{\kappa}}  \quad   \text{for} \quad \xi \in I_\kappa \subset [-C\kappa^{3/4}, C\kappa^{3/4}], \label{13_0809_1} \\
&  | \phi_\kappa'' (\xi)| = |\psi_\kappa''(\xi) |   \simeq |\xi^{-2/3}|  \quad  \text{for} \quad \xi \in  J_\kappa \setminus I_\kappa \subset [-C, C].\label{13_0809}
\end{align}
\end{subequations}
From \eqref{eq13_0809} and Proposition \ref{prop3.6} with the relation $\phi_\kappa'' = e^{\phi_\kappa} - \rho_\kappa$ ($\rho_\kappa \to 1$ as $|\xi| \to \infty$), it follows that 
for $p \in (1,3/2)$,
\[
 \| \phi_\kappa'' \|_{L^p (\mathbb{R})}^p \lesssim  \frac{c_\kappa^p \kappa^{3/4} }{\kappa^{p/2}} +  \int_{J_\kappa \setminus I_\kappa} \xi^{-2p/3} d\xi + 1 \lesssim  \int_{-C}^{C} \xi^{-2p/3} d\xi + 1 \lesssim 1,
 \]
 where the second inequality is due to \eqref{eq240716_01} and $c_\kappa \to c_0$ as $\kappa \to 0$.
Then by the Sobolev embedding, for any $\beta \in (0,1/3)$, $\|\phi_\kappa\|_{C^{1,\beta}(\mathbb{R})}$ is uniformly bounded.

We fix $\alpha \in (0, 1/3)$. By Arzela-Ascoli's theorem with a diagonal process, for any compact $K \subset \mathbb{R}$ there exists $\tilde{\phi}_0 \in C_{ loc }^{1,\alpha}(\mathbb{R})$ such that $\phi_\kappa \to \tilde{\phi}_0$ in $C^{1,\alpha}(K)$, up to subsequences.
With the fact that $\phi_\kappa$ decay exponentially (uniformly in $\kappa$) on $\mathbb{R} \setminus [C, C] \subset \mathbb{R} \setminus J_\kappa$, we obtain $\tilde{\phi}_0 \in C^{1, \alpha}(\mathbb{R})$ such that $\phi_\kappa \to \tilde{\phi}_0$ in $C^{1,\alpha}(\mathbb{R})$, up to subsequences. 

Moreover, by the relation $ \phi_\kappa'' = e^{\phi_\kappa} - H_\kappa^{-1}(\phi_\kappa) $
with Lemma \ref{H_K conv}, we also have      $\phi_\kappa'' \to \tilde{\phi}''_0$ in $\mathbb{R} \setminus (-\delta, \delta)$ for any $\delta > 0$.
This means that $\tilde{\phi}_0$ satisfies the following ODE: 
\[
\tilde{\phi}_0''=e^{\tilde{\phi}_0}-H_0^{-1}(\tilde{\phi}_0) \quad  \text{for } \xi \in \mathbb{R} \setminus \{0\}.
\]
Since 
\[
\tilde{\phi}_0(0) = \lim_{\kappa \to 0} \phi_\kappa(0) = \frac{c_0^2}{2},
\]
$\tilde{\phi}_0(\xi) \to 0$ as $|\xi| \to \infty$, and $\tilde{\phi}_0$ satisfies the above ODE, it must be $\tilde{\phi}_0 \equiv \phi_0$.
Therefore, $\phi_\kappa \to \phi_0$ in $C^{1,\alpha}(\mathbb{R})$ for $\alpha \in (0,1/3)$.
The first assertion is proved.

\vspace{15pt}
 
($2$)
Next, we prove \eqref{Eq_Thm3_2} for $1 \le p < 3/2$.
We proceed in a similar way to the proof of \eqref{Eq_Thm3}.  
In particular, for the convergence near the origin, due to  \eqref{lem3.3eq1} of Lemma \ref{lem3.3}, Proposition \ref{prop3.4} and Proposition \ref{prop3.6}, there exists $C>0$ such that
for any small $\delta>0$,
\[
\begin{split}
\|\rho_\kappa-1\|_{L^p((-\delta,\delta))} 
&  \le \|\rho_\kappa-1\|_{L^p(I_\kappa)} + \|\rho_\kappa-1\|_{L^p( (-\delta,\delta)\setminus I_\kappa)}   \\
& \le \|\rho_\kappa\|_{L^p(I_\kappa)} + \|\rho_\kappa\|_{L^p((-\delta,\delta)\setminus I_\kappa)}  + C\delta \\
& \lesssim c_\kappa \kappa^{-\frac{1}{2} + \frac{3}{4p}}  + \| \xi^{-2/3}\|_{L^p((-\delta,\delta)\setminus I_\kappa)} + C\delta\\
&  \lesssim \kappa^{-\frac{1}{2} + \frac{3}{4p}} + \delta^{\frac{1}{p}}
\end{split}
\]
for $\kappa \in [0, \kappa_0)$.
 Here, we have used $\rho_\kappa =H^{-1}_\kappa(\phi_\kappa)$ and $ \phi_\kappa'' = e^{\phi_\kappa} - H_\kappa^{-1}(\phi_\kappa) $.

On the other hand, by Lemma \ref{H_K conv}, $\rho_{\kappa} = H_{\kappa}^{-1}(\phi_{\kappa})$ uniformly converges to $\rho_0 = H_{0}^{-1}(\phi_{0})$ in $\mathbb{R} \setminus (-\delta, \delta)$ for any $\delta > 0$, as $\kappa \to 0$.
This with the uniform (in $\kappa$) boundedness of $\| \rho_\kappa - 1\|_{L^p((-\delta,\delta))}$ implies \eqref{Eq_Thm3_2}.

\vspace{15pt}

($3$) Finally, we prove \eqref{Eq_Thm3_3}. Fix $\beta \in (0, 2/3)$.
We prove that there is a constant $\kappa_0 > 0$ such that for all  $\kappa \in [0,\kappa_0]$, 
\begin{equation}\label{eq_0810}
\|1/\rho_\kappa \|_{C^{\beta}(\mathbb{R})} \le N,
\end{equation}
where  $N$ is a constant independent of $\kappa < \kappa_0$. 
If \eqref{eq_0810} holds true, \eqref{Eq_Thm3_3} follows as in the proof of \eqref{Eq_Thm3} by the relation $v_\kappa = c_\kappa(1 - 1/\rho_\kappa)$.

For simplicity, we set $ \nu_\kappa(\xi) := \rho_\kappa^{-1}(\xi) $. Then, $\sqrt{\kappa}/c_\kappa \le \nu_\kappa(\xi) < 1$ on $\xi \in \mathbb{R}$ with $\nu_\kappa(0) = \sqrt{\kappa}/c_\kappa \to 0$ as $\kappa \to 0$. Moreover, from \eqref{Def_H}, we have
\[
\phi_\kappa = \frac{c_\kappa^2}{2}\left(1 - \nu_\kappa^2 \right) + \kappa \log \nu_\kappa = \hat{H}_\kappa(\nu_\kappa),
\]
where $\hat{H}_\kappa(\tau) := c_\kappa^2(1 - \tau^2)/2 + \kappa \log \tau$ for $\tau \ge 0$. By direct calculation, we see that for $\xi \in \mathbb{R}$,
\begin{equation}
	\label{eq240805_01}
-c_\kappa^2 + \kappa \le \frac{d\hat{H}_\kappa }{d\tau}\left( \nu_\kappa \left( \xi \right) \right) = -c_\kappa^2\nu_\kappa \left( \xi \right) + \frac{\kappa}{ \nu_\kappa \left( \xi \right) } \le 0
\end{equation}
and
\begin{equation}
	\label{eq240805_02}
-2c_\kappa^2 \le  \frac{d^2\hat{H}_\kappa }{d\tau^2} \left( \nu_\kappa \left( \xi \right) \right) = -c_\kappa^2 - \frac{\kappa}{\nu_\kappa^2 \left( \xi \right) } \le -c_\kappa^2 - \kappa.
\end{equation}
In particular, $\hat{H}_\kappa(\tau)$ strictly decreases on $\tau\in[\sqrt{\kappa}/c_\kappa,1)$. 

For $\tau_1, \tau_2 \in [\sqrt{\kappa}/c_\kappa, 1)$ with $\tau_2 > \tau_1$, we observe that by \eqref{eq240805_01} and \eqref{eq240805_02},
\[
\begin{split}
\hat{H}_\kappa(\tau_2) - \hat{H}_\kappa(\tau_1) 
& = \frac{d\hat{H}_\kappa }{dt}(\tau_1)(\tau_2-\tau_1) + \frac{1}{2}\frac{d^2\hat{H}_\kappa }{dt^2} (\tilde{\tau}) \left( \tau_2 - \tau_1\right)^2 \\
& < \frac{1}{2}\frac{d^2\hat{H}_\kappa }{dt^2} (\tilde{\tau}) \left( \tau_2 - \tau_1\right)^2 < 0
\end{split}
\]
for some $ \tilde{\tau} \in [\tau_1 ,\tau_2]$. This with \eqref{eq240805_02} implies that there is a constant $N_1$,  independent of $\kappa$, such that
\begin{equation}\label{eq_0810_1}
\frac{ \left| \hat{H}_\kappa^{-1} \left( s_2 \right) - \hat{H}_\kappa^{-1} \left(s_1 \right) \right|}{|s_2-s_1|^{1/2}} \le N_1
\end{equation}
for $s_2 = \hat{H}_\kappa(\tau_2), s_1 = \hat{H}_\kappa(\tau_1) \in (0, \hat{H}_\kappa(\sqrt{\kappa}/c_\kappa)]$.

Now, we obtain that for $\xi_1,\xi_2 \in \mathbb{R}$,
\[
\begin{split}
& \frac{ \hat{H}_\kappa^{-1} \left( \phi_\kappa \left( \xi_1\right) \right) -  \hat{H}_\kappa^{-1} \left( \phi_\kappa \left( \xi_2 \right) \right)  }{|\xi_1-\xi_2|^{\beta}} \\
 = &\left( \frac{  \left|  \phi_\kappa\left( \xi_1 \right) - \phi_\kappa\left( \xi_2\right) \right|   }{|\xi_1-\xi_2|^{2\beta}} \right)^{1/2} \frac{ \hat{H}_\kappa^{-1} \left( \phi_\kappa \left( \xi_1\right) \right) -  \hat{H}_\kappa^{-1} \left( \phi_\kappa \left( \xi_2\right) \right)  }{     \left|  \phi_\kappa\left( \xi_1 \right) - \phi_\kappa\left( \xi_2\right) \right|^{1/2}      } \\
 \leq & \, C \frac{ \hat{H}_\kappa^{-1} \left( \phi_\kappa \left( \xi_1\right) \right) -  \hat{H}_\kappa^{-1} \left( \phi_\kappa \left( \xi_2\right) \right)  }{     \left|  \phi_\kappa\left( \xi_1 \right) - \phi_\kappa\left( \xi_2\right) \right|^{1/2}      } 
  \le  CN_1,
\end{split}
\]
provided  $\beta \in (0, 2/3)$,
where  the first inequality is due to the first assersion (1), i.e., the uniform boundedness of $\| \phi_\kappa \|_{C^{1, \alpha}(\mathbb{R})}$ in $\kappa$ for $\alpha \in (0, 1/3)$ and the second inequality is from \eqref{eq_0810_1}. This proves \eqref{eq_0810}, and therefore, the theorem is proved. 
\end{proof}

\section{Numerical solutions to the Euler-Poisson system}\label{S4}
In this section, we present some numerical solutions to the pressureless Euler-Poisson system, i.e., \eqref{EP2} with $\kappa=0$. 
We consider smooth and exponentially decaying initial data, and we numerically investigate  whether there are finite-time $C^1$ blow-up solutions to \eqref{EP2} with $\kappa=0$ whose asymptotic blow-up profiles resemble the peaked solitary waves of \eqref{EP2} with $\kappa=0$.

We follow \cite{Satt} (see also \cite{HNS}) and employ the implicit pseudo-spectral scheme with $\Delta x = L/2^9$ on periodic spatial domains $[-L,L]$ and the Crank–Nicolson method with $\Delta t = 0.001$ is applied for time marching. 

Let us first briefly explain one possible type of singularity (\textit{shock-like}) in the pressureless Euler-Poisson system. Sufficient conditions for initial data that lead to the $C^1$ blow-up are studied in \cite{BCK,CKKT,Liu} (see also the Appendix in \cite{BKK2}).
In particular, numerical solutions presented in \cite{BCK} show that the velocity function $v$ exhibits \textit{shock-like} singularity (see Figure \ref{FigNume6}, for instance).
To study a more detailed blow-up profile, in \cite{BKK2}, the blow-up solution for the pressureless Euler-Poisson system is constructed from the initial data satisfying
\[
\partial_x v|_{t=0} \ll -1, \quad \partial_x^2 v|_{t=0} = 0, \quad \partial_x^3 v|_{t=0} \gg 1
\]
at some point $x=x_0$.
In this construction, suitable self-similar and modulation variables are introduced, considering the pressureless Euler-Poisson system as a perturbation of the Burgers equation, so that the inflection point of $v$ becomes the blow-up location.
When it blows up at $t=T_\ast$, then $v$ exhibits $C^{1/3}$ regularity, and for $\rho$, it holds that 
\begin{equation}\label{eq0903}
\left| \rho \left( T_\ast, x \right) -1 \right| \sim \frac{1}{ \left( x-x_* \right)^{2/3}}
\end{equation}
near the blow-up location $x_\ast$. We observe that the asymptotic behavior \eqref{eq0903} of $\rho$  is identical to that of the peaked solitary wave in \eqref{Thm2_1}.

In order to further illustrate, in Figure \ref{FigNume6}, we present a numerical solution to the pressureless Euler-Poisson system \eqref{EP2} with the initial data $(\rho,u)|_{t=0}=(1,3e^{-x^2})$.
Wave steepening of $v$ is observed: one front of $v$ becomes steeper, and the gradient of $v$ tends to $-\infty$ in finite time $T_\ast \sim 4$.
To the left (resp. the right) of the  singularity,  $v$ is concave down (resp. up). This corresponds to the formation of a shock-like (Burgers-type) singularity.

\begin{figure}[t]
\begin{center}
\includegraphics[width=0.85\linewidth]{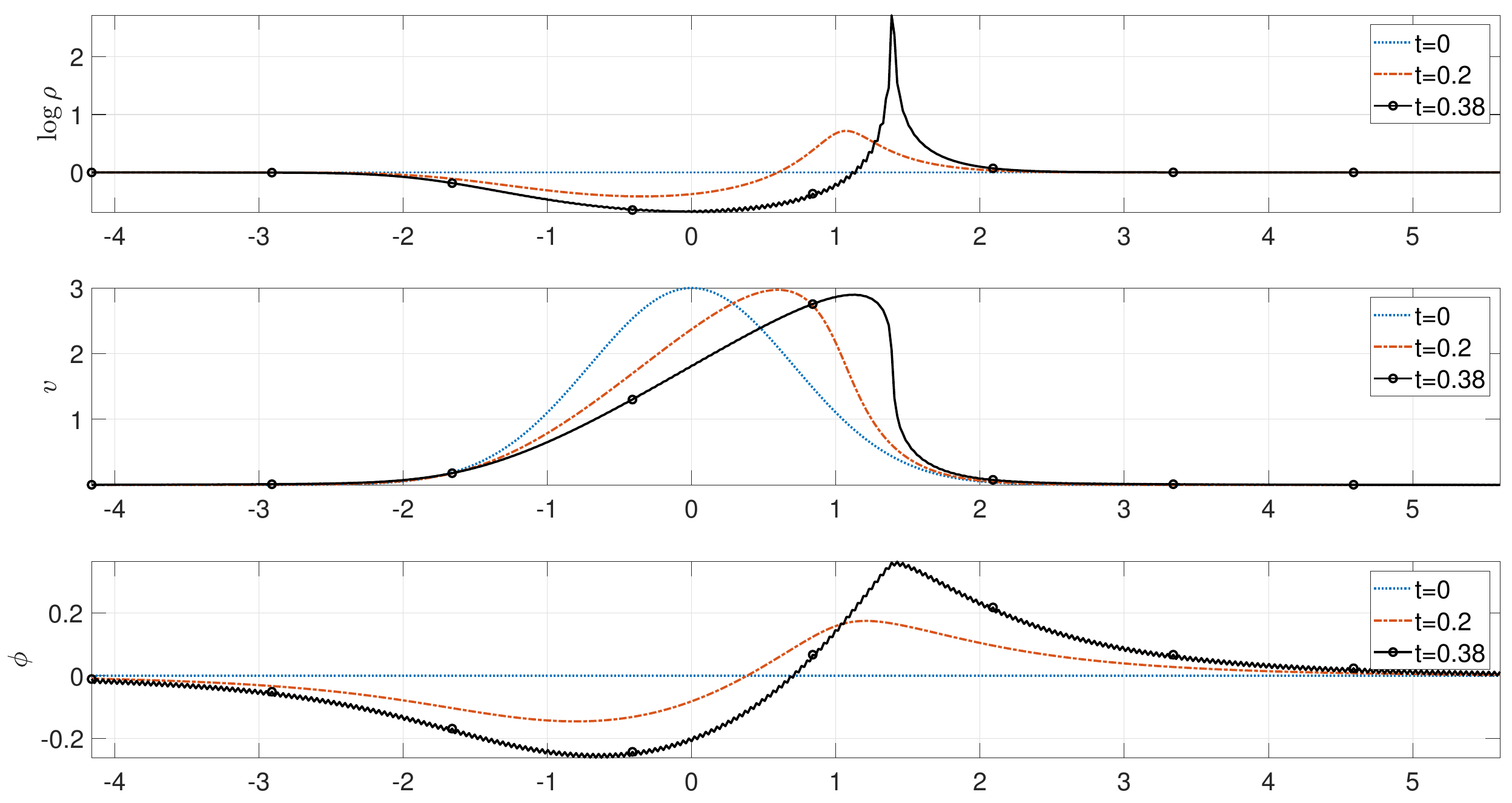}
\end{center}
\caption{Numerical solution of the pressureless Euler-Poisson system (\eqref{EP2} with $\kappa=0$) for the initial data $\rho|_{t=0}=1$ and $v|_{t=0}=  3e^{-x^2}$ on the interval $[-10,10]$.  }
\label{FigNume6}
\end{figure}

\begin{figure}[h]
\begin{center}
\includegraphics[width=0.85\linewidth]{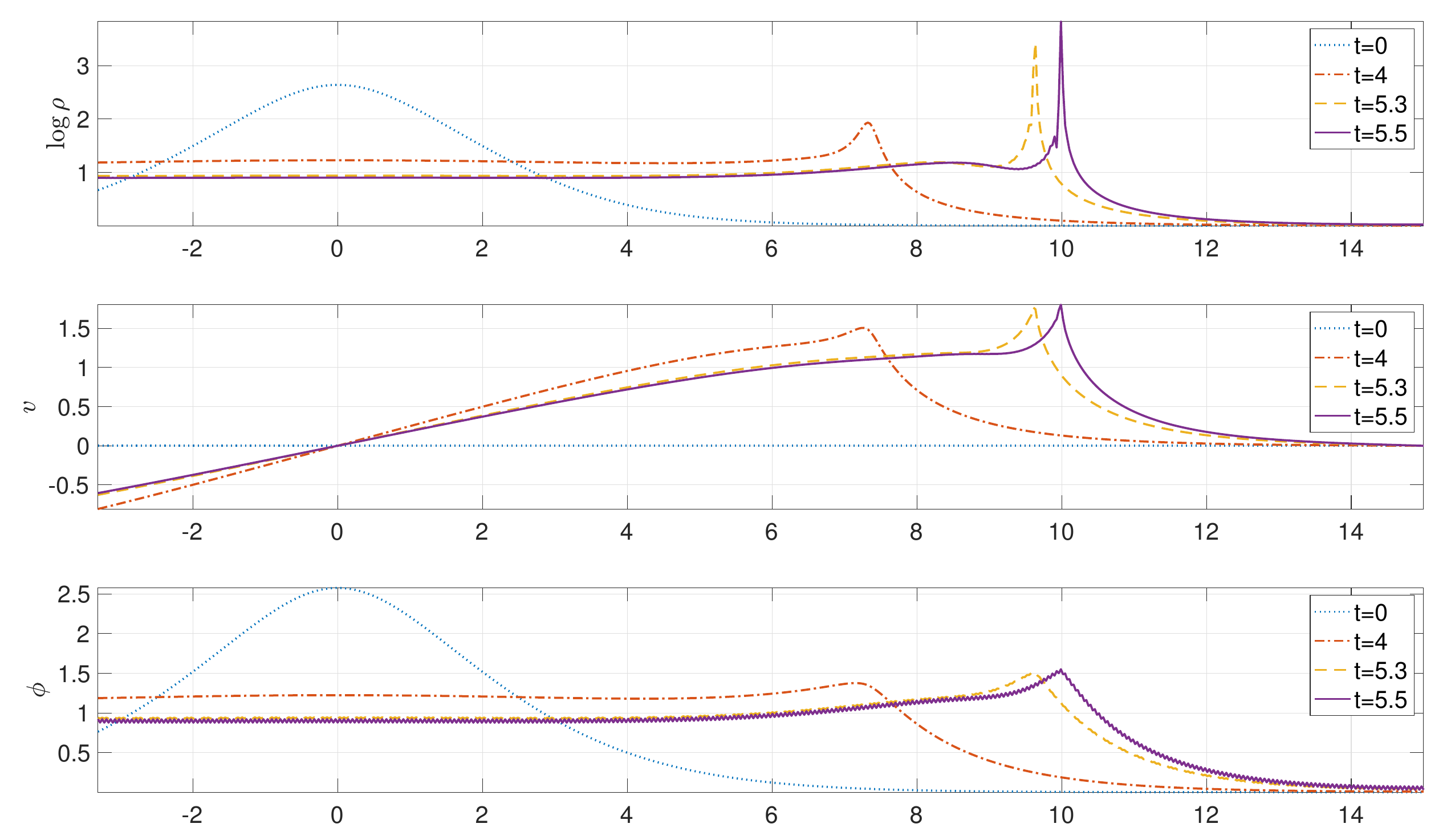}
\end{center}
\caption{Numerical solution of the pressureless Euler-Poisson system (\eqref{EP2} with $\kappa=0$) for the initial data $\rho|_{t=0}=1 + 13 \,\text{sech}(x)$ and $v|_{t=0}=0$ on the interval $[-15,15]$. }
\label{FigNume3}
\end{figure}

On the other hand, in Figure \ref{FigNume3}, where the initial data $\rho|_{t=0}=1 + 13 \,\text{sech}(x)$ and $v|_{t=0}=0$ are considered, we observe a very different behavior of the solution.
At the blow-up location $x_\ast \sim  10$,  $v$ resembles the shape of the peaked solitary wave profiles for $v$, rather than that of the blow-up profile of the Burgers equation. Specifically, to the left (resp. the right) of the blow-up location $x_\ast \sim  10$, the sign of $\partial_x v$ is positive (resp. negative), and $v$ is concave up near $x_\ast \sim  10$.
Compare Figure \ref{FigNume6} and Figure \ref{FigNume3} .

It is important to note that the initial data for the solution in Figure \ref{FigNume3} do not fit into the class of initial conditions leading to the $C^1$ blow-up as studied in \cite{BCK, BKK2, CKKT,Liu}, in which the initial data must satisfy that either $|\partial_x v|$ is sufficiently large or $\rho$ is sufficiently close to $0$ at some point $x$ (see Figure 1 in \cite{CKKT}).
Thus,  this represents a novel blow-up scenario.

In fact, the proofs of blow-up in these works rely on the uniform (in $x$ and $t$) boundedness of $\phi$ obtained from the conserved energy, and they do not account for the dispersive nature in the dynamics of \eqref{EP2}.
We recall that the peaked solitary waves can exist due to the balance between the nonlinear transport and the dispersive effect.

Summarizing our discussion, we propose the following conjecture on the pressureless Euler-Poisson system, i.e., \eqref{EP2} with $\kappa=0$.
For the initial data $(\rho-1,v)|_{t=0} \in H^{s}(\mathbb{R})\times H^{s+1}(\mathbb{R})$, $s>3/2$, the classical solution to \eqref{EP2} with $\kappa=0$ exists locally in time \cite[Theorem 10.1]{LLS}.\\

\textbf{Conjecture:} \textit{For some $s>3/2$, there exists a class of initial data $(\rho-1,v)|_{t=0} \in H^{s}(\mathbb{R})\times H^{s+1}(\mathbb{R})$, which includes $(\tilde{\rho},0)$ where $\tilde{\rho}(x)>1$ for all $x\in \mathbb{R}$, and $\tilde{\rho}(x)$ is sufficiently large at some point $x$, such that the following hold:
\begin{enumerate}
\item the maximal existence time $T_\ast$ of the $C^1$ solution  $(\rho,v)$ is finite;
\item at the blow-up time $T_\ast$, $v$ is concave up to the left (and right) of the blow-up location $x_\ast$;  
\item at the blow-up time $T_\ast$, the asymptotic behaviors of $v$ and $\rho$ follow the forms $(x-x_\ast)^{2/3}$ and $(x-x_\ast)^{-2/3}$, respectively.
\end{enumerate}
} 
\bigskip

The asymptotic behaviors conjectured above are based on the behaviors of the peaked solitary waves in \eqref{Thm2_1} and numerical solutions shown in Figure \ref{FigNume3}.
As in \cite{KS,RLSS}, where a method from \cite{SSF} was adopted to track singularities arising in some dispersive equations such as the Whitham and the Camassa-Holm equations, more careful numerical study of the form of the singularities in the pressureless model will be needed. 

We finish this section with a remark.
For the isothermal Euler-Poisson system, i.e., \eqref{EP2} with $\kappa>0$, $C^1$ solutions are not expected to become Lipschitz prior to the derivative blow-up, due to the local existence theorem \cite[Theorem 10.4]{LLS} and its corollary \cite[Lemma 2.3]{BKK}.
More specifically, for $(\rho-1,v)|_{t=0}\in H^s(\mathbb{R})\times H^s(\mathbb{R})$ with $s > 3/2$, there is $T>0$ such that $(\rho-1,v)\in C([0,T];H^s(\mathbb{R})\times H^s(\mathbb{R}))$.
If $\lim_{t \to T}\|\partial_x(\rho,v)(t,\cdot)\|_{L^\infty(\mathbb{R})}<\infty$, then the solution can be continued beyond $T$. Hence, by the Sobolev embedding, the solution $(\rho-1,v)$ to the isothermal model cannot become Lipschitz (leaving $C^1$ class) in finite time for the initial data in $H^{s}(\mathbb{R})$, $s > 3/2$. We refer to \cite{BCK,BKK} for the finite-time $C^1$ blow-up of the isothermal model. To the best of our knowledge, the global existence of smooth solutions to \eqref{EP2} remains an open question.


  \section*{Acknowledgments}
J. Bae was supported by the National Research Foundation of Korea grant funded by the Korea government (MSIT) (NRF-2022R1C1C2005658). 
S.H. Moon was supported by the National Research Foundation of Korea grant funded by the Korea government (MSIT) (NRF-2022R1C1C2009892).
K. Woo has been supported by the National Research Foundation of Korea (NRF) grant funded by the Korean government (MSIT) (No.2022R1A4A1032094).
 
\subsection*{Data availability statement}
All data that support the findings of this study are included within the article (and any supplementary files).

\subsection*{Conflict of interest}
The authors declare that they have no conflict of interest.

\section{Appendix} 
\begin{lemma}
	\label{z_K est}
Consider the equations \eqref{Aux3 lem-1} and \eqref{Eq z0 Cold-1}. Then, the following   hold:
\begin{enumerate}[$(1)$]
\item For each $\kappa>0$, the equation \eqref{Aux3 lem-1} has exactly two positive solutions $z=\sqrt{\kappa}$ and $z=c_\kappa$. In particular, $c_\kappa>\sqrt{1+ \kappa}$.
\item  The equation \eqref{Eq z0 Cold-1} has a unique positive solution $z=c_0$, and it holds that $c_0>1$.
\item $c_\kappa \to c_0$ as $\kappa \to 0$.
\item For all sufficiently small $\kappa>0$, it holds that  $c_\kappa < c_0$ and $c_0- c_\kappa = O(\sqrt{\kappa})$.
\end{enumerate}
\end{lemma}

\begin{proof}
(1) Taking logarithm on both sides of \eqref{Aux3 lem-1}, we define
$$
f_\kappa(z):=\kappa \log z - \kappa\log\sqrt{\kappa} + \log \left( (z-\sqrt{\kappa})^2 +1\right) -  (z^2 - \kappa)/2.
$$
Since
\[
\frac{df_\kappa}{dz} = \frac{-(z-\sqrt{\kappa})^2}{z\big((z-\sqrt{\kappa})^2 + 1 \big)}\left( z + \sqrt{1+\kappa} \right)\left( z - \sqrt{1+\kappa} \right),
\]
$f_\kappa(\sqrt{\kappa}) = 0$ and $\lim_{z \to +\infty}f_\kappa(z) = -\infty$, it is straightforward to see that $f_\kappa$ has exactly two positive roots $z=\sqrt{\kappa}$ and $z=c_\kappa$, and $c_\kappa$ satisfies $\sqrt{1+\kappa} < c_\kappa$.
This proves the first statement.

(2) In a similar fashion, one can show the second statement by introducing 
\[
f_0(z):=\log \left(z^2+1\right) -   z^2/2
\]
and showing that $f_0$ has a unique positive zero $c_0$.
We omit the details. 

(3) To show the third statement, we note that $f_\kappa \to f_0$ uniformly in $z$ on any compact interval in $(0,\infty)$.
Since $c_0$ is a unique positive zero of $f_0$,  we conclude that $c_\kappa \to c_0$ as $\kappa \to 0$. 

(4) Now we prove that $c_\kappa < c_0$ for all sufficiently small  $\kappa>0$.
It is straightforward to check that  $f_\kappa(z) >0$ for $z \in (\sqrt{\kappa},c_\kappa)$ and $f_\kappa(z)<0$ for $z \in (c_\kappa,\infty)$.
Since $\log{(1+x)} \simeq x$ near $0$, we have that for sufficiently small $\kappa>0$, 
\begin{equation}\label{f_ka1}
\begin{split}
f_\kappa(c_0)
& = \kappa\log c_0  -  \frac{\kappa}{2}\log \kappa + \log \left(c_0^2 -2\sqrt{\kappa}c_0+\kappa^2+1\right) - \frac12 \left( c_0^2 - \kappa \right)\\
& =  \kappa\log c_0 -  \frac{\kappa}{2}\log \kappa + \log (c_0^2 + 1)  -\frac{ 2c_0\sqrt{\kappa}}{c_0^2 + 1} (1+ o(1))   -\frac12 c_0^2 + \frac{\kappa}{2}  \\
 & =   -\frac{ 2c_0\sqrt{\kappa}}{c_0^2 + 1} (1+ o(1)) < 0,
\end{split}
\end{equation}
where we use the fact that $f_0(c_0)= \log \left(c_0^2+1\right) -   c_0^2/2 =0$.
Hence, it follows that  $c_\kappa < c_0$ for all sufficiently small  $\kappa>0$. 
Moreover, since $f_0'(c_0) < 0$ and $f_\kappa \to f_0$ in $C^1\left( K \right)$ for any compact set $K \subset \left( 0,\infty \right)$, there is $\delta>0$ such that $f_\kappa'(z) <-\delta$ for all $z$ sufficiently close to $c_0$.
Then, by the fundamental theorem of calculus,
\[
-f_\kappa(c_0)  = -f_\kappa(c_0) + f_\kappa(c_\kappa)  = \int_{c_\kappa}^{c_0} -f_\kappa'(z)\,dz  > \delta(c_0-c_\kappa),
\]
and from \eqref{f_ka1}, it follows that $c_0-c_\kappa= O(\sqrt{\kappa})$. We complete the proof.
\end{proof}

\bibliographystyle{amsplain}

 \end{document}